\documentclass[12pt]{amsart}

\usepackage{amsmath}
\usepackage{amssymb}
\usepackage{amsthm}
\usepackage{xcolor}
\usepackage{bm}
\usepackage{graphicx}
\usepackage{pifont}
\usepackage{mathrsfs}
\usepackage{epsfig,verbatim}
\usepackage[pagewise]{lineno}\linenumbers

\newcommand\be{\begin{equation}}
\newcommand\ee{\end{equation}}
\newcommand\bea{\begin{eqnarray}}
\newcommand\eea{\end{eqnarray}}
\newcommand\beaa{\begin{eqnarray*}}
\newcommand\eeaa{\end{eqnarray*}}
\newcommand\beba{\begin{equation}\left\{\begin{array}{rcl}}
\newcommand\eeba{\end{array}\right.\end{equation}}
\newcommand\bebaa{\begin{equation*}\left\{\begin{array}{rcl}}
\newcommand\eebaa{\end{array}\right.\end{equation*}}

\newcommand\bR{{\mathbb{R}}}

\newcommand\cK{{\mathcal{K}}}

\newcommand\cX{{\mathcal{X}}}

\newcommand{\eps}{\epsilon}

\newcommand{\dint}{\displaystyle \int}
\newcommand{\dsum}{\displaystyle \sum}

\newcommand{\dlim}{\displaystyle \lim}

\newtheorem{theorem}{Theorem}[section]
\newtheorem{lemma}[theorem]{Lemma}

\newtheorem{example}[theorem]{Example}
\newtheorem{definition}[theorem]{Definition}
\newtheorem{remark}[theorem]{Remark}
\newtheorem{proposition}[theorem]{Proposition}

\numberwithin{equation}{section}

\begin{document}
\nolinenumbers

\title[Asymptotic traveling waves]{Global front solutions converging to an asymptotic traveling wave in a time-fractional Amari model}

\author[H. Ishii]{Hiroshi Ishii}
\address{Research Center of Mathematics for Social Creativity, Research Institute for Electronic Science, Hokkaido University, Hokkaido, 060-0812, Japan}
\email{hiroshi.ishii@es.hokudai.ac.jp}

\thanks{Date: \today. Corresponding author: Hiroshi Ishii}


\begin{abstract}
We study front propagation in a time-fractional Amari neural field equation with a Caputo derivative and a Heaviside firing-rate function. 
For a class of nonnegative even connectivity kernels and strictly decreasing $C^1$ front-like initial data, we reduce the Cauchy problem to a nonlinear Volterra equation for the threshold location. 
We prove the existence and uniqueness of a global front solution, determine the short-time behavior of its threshold location, and identify its asymptotic speed. 
We also construct the associated asymptotic traveling-wave profile and establish uniqueness of the speed and profile, up to translation, in the nonzero-speed case. 
Using a rigidity argument, we prove the asymptotic linearity of the local increments of the threshold location and obtain uniform convergence to the profile in the threshold-centered frame. 
Finally, we characterize the kernel-dependent tail asymptotics of the wave profile.
\end{abstract}

\maketitle
\setcounter{tocdepth}{3}

\section{Introduction}\label{sec:intro}

Evolution equations involving Caputo time-fractional derivatives have recently attracted attention as mathematical models for propagation phenomena in media with memory. 
A representative example is the time-fractional diffusion equation, which can be derived from a continuous-time random walk with a heavy-tailed waiting-time distribution as a subdiffusion model \cite{MK}.
Time-fractional equations have also been used to model phenomena such as the motion of a large thin plate in a Newtonian fluid \cite{TB}, chemical reaction systems with subdiffusive media \cite{HLW}, tumor growth \cite{FK}, and so on.

Time-fractional dynamics have also been incorporated into neural field theory. 
Neural field theory describes the spatiotemporal dynamics of large neuronal populations at a macroscopic level, using population-level variables such as membrane potentials and firing rates. 
This framework has led to a broad class of integro-differential equations, commonly referred to as neural field equations, for modeling excitation patterns in spatially distributed neuronal populations (e.g. \cite{Amari, Bress, CPWT}). 
A fundamental example is the Amari model \cite{Amari}. 
On the real line, the classical Amari model takes the form
\beaa
\partial_{t} u(t,x) = \int_{\bR} K(x-y) F(u(t,y)) dy - u(t,x) \quad (t>0,\ x\in\bR).
\eeaa
Here, $u(t,x)$ denotes the activity of a neuronal population at time $t$ and position $x$.
The connectivity kernel $K$ represents the strength of synaptic connectivity as a function of spatial distance, while $F$ describes the firing-rate response. 
A standard choice for $F$ is the Heaviside function 
\beaa 
    F(u)=H(u-u_T), 
\eeaa 
where $u_T>0$ denotes the firing threshold. 
The Amari model and its related models generate a variety of spatiotemporal dynamics. 
 
To incorporate memory effects into neural field dynamics, the following time-fractional Amari model was proposed in \cite{Gonzalez}:
\be\label{eq:main}
\partial^{\alpha}_{t} u(t,x) = \int_{\bR} K(x-y) H(u(t,y)-u_T) dy - u(t,x) \quad (t>0,\ x\in\bR) \tag{TfA}
\ee
where $\alpha\in(0,1)$ and $\partial_t^\alpha$ denotes the Caputo-type fractional derivative
\beaa
\partial^{\alpha}_{t} u (t,x):= \dfrac{1}{\Gamma(1-\alpha)} \int^{t}_{0}\dfrac{\partial u}{\partial s}(s,x) \dfrac{ds}{(t-s)^{\alpha}}.
\eeaa
In \cite{Gonzalez}, front propagation was investigated for the exponential kernel 
\beaa
    K(x) = \dfrac{1}{2\sigma} \exp\left( - \dfrac{|x|}{\sigma} \right)\ (\sigma>0)
\eeaa
with the aim of clarifying the effects of the time-fractional derivative on front propagation.

A fundamental difficulty in the analysis of propagation phenomena governed by Equation~\eqref{eq:main} is that the usual traveling-wave ansatz does not reduce the equation to an autonomous profile equation. 
Indeed, suppose that 
\beaa
    u(t,x)=\phi(x-ct)
\eeaa 
for some speed $c\in\bR$, and set $z:=x-ct$.
Then 
\be\label{eq:Caputo}
\partial^{\alpha}_{t} u(t,x) = -\dfrac{c}{\Gamma(1-\alpha)} \int^{t}_{0} \dfrac{\phi'(x-cs)}{(t-s)^{\alpha}}\, ds= -\dfrac{c}{\Gamma(1-\alpha)} \int^{t}_{0} \dfrac{\phi'(z+cs)}{s^{\alpha}}\, ds
\ee
Thus, even in the moving coordinate $z$, the right-hand side of \eqref{eq:Caputo} retains an explicit dependence on $t$. 
Consequently, unlike in the classical case, the traveling-wave ansatz does not transform Equation~\eqref{eq:main} into a time-independent profile equation. 
In this sense, the standard theory of traveling waves cannot be applied directly to the time-fractional problem.

For this reason, several previous studies have introduced an asymptotic traveling wave by formally assuming that
\beaa
     u(t,x) \simeq \phi(x-ct) \quad (t\to+\infty)
\eeaa
and replacing the upper limit in \eqref{eq:Caputo} by $+\infty$. 
More precisely, for $c\neq0$, define 
\beaa
    D^{\alpha}_{c} \phi(z) := 
    \begin{cases}
        -\dfrac{c^{\alpha}}{\Gamma(1-\alpha)} \dint^{+\infty}_{0} \dfrac{\phi'(z+s)}{s^{\alpha}}\, ds &(c>0), \vspace{3mm}\\ 
        \dfrac{|c|^{\alpha}}{\Gamma(1-\alpha)} \dint^{+\infty}_{0} \dfrac{\phi'(z-s)}{s^{\alpha}}\,ds &(c<0).
    \end{cases}
\eeaa
A pair $(c,\phi)$ satisfying the formal limiting equation 
\be\label{eq:atw}
    D^\alpha_{c}\phi(z) = \int_{\bR} K(z-y)H(\phi(y)-u_T)\,dy - \phi(z) \tag{ATW}
\ee
is referred to as an asymptotic traveling wave. 
Such waves have been studied for the time-fractional Fisher--KPP equation \cite{IH}, time-fractional Allen--Cahn equations \cite{NVN,VNN}, and the time-fractional Amari model \cite{Gonzalez}.
These studies identify characteristic propagation speeds and profiles and thereby clarify how the fractional time derivative affects front propagation.

The derivation of Equation~\eqref{eq:atw}, however, is based on a formal long-time approximation. 
As pointed out in \cite{FDPDE}, it does not by itself guarantee the existence of a solution of the original Cauchy problem that converges to the resulting asymptotic traveling wave. 
To the best of our knowledge, for the Cauchy problem of the Caputo-type neural field equation considered here, it has not previously been proved that an asymptotic traveling wave is realized as the long-time limit of an actual solution.

In the author's previous work \cite{IH}, asymptotic traveling waves for a time-fractional Fisher--KPP equation were investigated. 
An expected minimal propagation speed was derived, and numerical simulations supported an analogue of the classical KPP speed-selection principle. 
Nevertheless, convergence of solutions of the Cauchy problem to the asymptotic traveling waves remained unresolved. 
Establishing such a convergence mechanism is therefore an important step toward a rigorous theory of front propagation in time-fractional evolution equations.

The aim of this paper is to establish such a mechanism for the time-fractional Amari model \eqref{eq:main}.
We consider monotone front-like initial data and exploit the Heaviside firing-rate function to describe the front position by a single threshold location $x^*(t)$. 
Consequently, the original Cauchy problem can be reduced to a nonlinear Volterra equation for $x^*(t)$.
This reduction makes it possible to construct a global front solution and to analyze its propagation directly.

Within this framework, we construct a unique global front solution, determine the short- and long-time behavior of the threshold location, and identify the propagation speed. 
We also construct the corresponding asymptotic traveling wave and prove its uniqueness up to spatial translation in the nonzero-speed case. 
Finally, we show that the global front solution converges uniformly, in the moving frame attached to $x^*(t)$, to the asymptotic traveling-wave profile.

A key feature of our approach is that it does not rely on an explicit formula for the connectivity kernel. 
In contrast to \cite{Gonzalez}, where the construction uses the special structure of an exponential kernel, our arguments apply to a broad class of nonnegative, even, bounded, continuous, and integrable kernels.
Thus, the global construction of front solutions, speed selection, the construction of asymptotic traveling waves, and convergence to the profile are treated within a single rigorous framework.

The remainder of this paper is organized as follows: In Section~\ref{sec:main}, we state the assumptions and the main results. 
We construct the global front solution and analyze the corresponding threshold location in Section~\ref{sec:global}.
In Section~\ref{sec:atw}, we establish the existence and uniqueness of asymptotic traveling waves. 
In Section~\ref{sec:convergence}, we prove convergence of the global front solution to the selected asymptotic traveling wave. 
We investigate the kernel dependence of the wave profile in Section~\ref{sec:ker}. 
Finally, Section~\ref{sec:dis} discusses several remaining problems.

\bigskip
\section{Main results}\label{sec:main}

In this section, we introduce the assumptions used throughout the paper and state our main results.

We assume that the connectivity kernel $K$ satisfies 
\be\label{condi:kernel}
K\in L^1(\bR)\cap BC(\bR), \qquad K(0)>0,\qquad K\geq 0,\qquad K(x)=K(-x)\ (\forall x\in\bR). \tag{K}
\ee
Here, $BC(\bR)$ denotes the space of bounded continuous functions on $\bR$.
We set 
\beaa
\kappa:=\int_{\bR} K(y)\,dy >0
\eeaa
and assume that the firing threshold $u_T$ satisfies
\be\label{condi:bistable}
 u_{T}\in (0,\kappa). \tag{B}
\ee
Under \eqref{condi:bistable}, the spatially homogeneous equation of \eqref{eq:main} is
\beaa
\partial^{\alpha}_{t} u^* =  \kappa H(u^*-u_T) - u^*
\eeaa
which has the two stable equilibria $0$ and $\kappa$.

We consider the initial condition
\be\label{init:main}
    u(0,x) = u_0(x)\quad (x\in\bR), \tag{IC}
\ee
where $u_0$ is assumed to satisfy
\be\label{condi:init}
    \begin{cases}
        u_0\in C^1(\bR),\quad u'_0<0,\\
        u_0(0)=u_T,\quad u_0(-\infty)= \kappa, \quad u_0(+\infty)=0.
    \end{cases}\tag{I}
\ee

We consider the mild solution $u\in C([0,+\infty)\times\bR)$ to \eqref{eq:main} with \eqref{init:main} (see Section~\ref{subsec:int-rep} for details).
We first define the class of front solutions considered in this paper.
\begin{definition}
   A mild solution $u\in C([0,+\infty)\times\bR)$ is a front solution to Equation~\eqref{eq:main} with \eqref{init:main} if there exists $x^*(t)\in C[0,+\infty)$ such that the solution satisfies
   \be\label{sol:front}
    u(t,x) 
    \begin{cases}
    >u_T &(x<x^*(t)), \\
    =u_T &(x=x^*(t)), \\
    <u_T &(x>x^*(t))
    \end{cases}
    \ee
    for any $t\ge 0$. We refer to $x^*(t)$ as the \textbf{threshold location}.
\end{definition}

\medskip
To state the main results, we introduce the tail integral of the connectivity kernel,
\beaa
\cK(x):= \int^{+\infty}_{x} K(y)\,dy
\eeaa
It follows from \eqref{condi:kernel} that $\cK$ is non-increasing and satisfies 
\beaa
    0\leq \cK(x)\leq\kappa, \qquad \cK(-\infty)=\kappa, \qquad \cK(+\infty)=0. 
\eeaa
Moreover, the evenness of $K$ gives 
\beaa
    \cK(0)=\frac{\kappa}{2}. 
\eeaa

For $\alpha,\beta>0$, let 
\beaa
     E_{\alpha,\beta}(z) := \dsum^{\infty}_{n=0} \dfrac{z^{n}}{\Gamma(\alpha n + \beta)}
\eeaa
be the two-parameter Mittag--Leffler function, and write $ E_\alpha(z):=E_{\alpha,1}(z)$.
We also set 
\beaa
g_{\alpha}(t) := t^{\alpha-1} E_{\alpha,\alpha}(-t^{\alpha})
\eeaa
For $\alpha\in(0,1)$, the function $g_\alpha$ is positive and integrable on $(0,\infty)$. 
Further properties of $E_\alpha$ and $g_\alpha$ used in the proofs are collected in Section~\ref{subsec:int-rep}.

Our first result establishes the global existence and uniqueness of a front solution and determines the short- and long-time behavior of its threshold location.

\begin{theorem}\label{thm:global}
    Under Conditions \eqref{condi:kernel}, \eqref{condi:bistable} and \eqref{condi:init}, there exists a unique front solution to Equation \eqref{eq:main} with \eqref{init:main}.
    It satisfies
    \beaa
        \dlim_{x\to-\infty} u(t,x) = \kappa,\quad \dlim_{x\to+\infty} u(t,x) = 0
    \eeaa
    for any $t\ge 0$.
    Moreover, the threshold location $x^*(t)$ satisfies
    \be\label{xs:short}
        \dlim_{t\to +0} \dfrac{x^*(t)}{t^{\alpha}} = -\frac{\kappa-2u_T}{2u_0'(0)\Gamma(1+\alpha)}
    \ee
    and
    \be\label{xs:long}
        \dlim_{t\to +\infty} \dfrac{x^*(t)}{t} = c^*,
    \ee
    where $c^*$ is the unique solution of 
    \be\label{eq:speed}
        u_T = \int^{+\infty}_{0} g_{\alpha}(r) \cK(cr) dr =: G_{\alpha}(c). \tag{S}
    \ee
\end{theorem}

\medskip
The function $G_{\alpha}:\bR\to\bR$ is continuous and strictly decreasing from $\kappa$ to $0$.
Since $\cK(0)=\kappa/2$ and $\int^{+\infty}_{0} g_\alpha(r)\,dr=1$, we can see that $G_{\alpha}(0)=\kappa/2$. Consequently, 
\beaa
    c^*>0 \quad\Longleftrightarrow\quad \kappa > 2u_T,
\eeaa
and analogous equivalences hold for $c^*=0$ and $c^*<0$.

The short-time expansion \eqref{xs:short} shows that, when $\kappa\neq 2u_T$, the initial displacement of the interface is of order $t^\alpha$.
This fractional short-time scale is distinct from the linear scale associated with classical traveling fronts. 
The proof of Theorem~\ref{thm:global} is given in Section~\ref{sec:global}.

We next formulate the asymptotic traveling-wave problem.
\begin{definition}\label{def:atw}
    A pair $(c,\phi)\in \bR\times BC^1(\bR)$ is called an 
    \textbf{asymptotic traveling wave} to Equation~\eqref{eq:main} if it satisfies Equation~\eqref{eq:atw}.
    Here, we define $D^{\alpha}_c \phi(z) \equiv 0$ when $c=0$.
\end{definition}

\medskip
The next theorem gives an explicit representation of the asymptotic traveling wave and establishes its uniqueness in the nonzero-speed case.
\begin{theorem}\label{thm:atw}
    Under Conditions \eqref{condi:kernel} and \eqref{condi:bistable}, let 
    \be\label{phi}
        \phi(z;c) := \int^{+\infty}_{0} g_{\alpha}(r)\cK(z + c r) dr.
    \ee
    Then, $(c^*, \phi(z;c^*))$ is an asymptotic traveling wave to Equation~\eqref{eq:main} with the non-increasing profile satisfying
    \be\label{boundary}
        \phi(-\infty) = \kappa,\quad \phi(+\infty) = 0.
    \ee
    Here, $c^*$ is the unique solution of \eqref{eq:speed}.
    
    Under the condition $\kappa \neq 2u_T$, suppose that $(c,\psi)$ is an asymptotic traveling wave to Equation~\eqref{eq:main} satisfying \eqref{boundary}.
    Then we have $c=c^*$ and $\psi(z) = \phi(z-z_0;c^*)$ for some $z_0\in\bR$.
\end{theorem}

The explicit representation \eqref{phi} also yields the symmetry relation 
\be\label{inverse}
    \phi(z;c) = \kappa-\phi(-z;-c).
\ee
When $\kappa=2u_T$, equivalently $c^*=0$, uniqueness need not hold in general; see Example~\ref{exam:non-unique}.
The proof of Theorem~\ref{thm:atw} is given in Section~\ref{sec:atw}.
The kernel dependence of the wave profile is also investigated in Section~\ref{sec:ker}.

Our final main result shows that the asymptotic traveling wave constructed above is realized as the long-time limit of the front solution from Theorem~\ref{thm:global}.

\begin{theorem}\label{thm:converge}
    Under Conditions \eqref{condi:kernel}, \eqref{condi:bistable} and \eqref{condi:init}, the front solution $u(t,x)$ to Equation~\eqref{eq:main} with \eqref{init:main} satisfies
    \beaa
        \dlim_{t\to +\infty}\sup_{z\in\bR} |u(t,x^*(t)+z) - \phi(z;c^*)| = 0,
    \eeaa
    where $\phi(z;c)$ is defined by \eqref{phi}, and $c^*$ is the unique solution of Equation~\eqref{eq:speed}.
\end{theorem}

Theorem~\ref{thm:converge} gives a rigorous connection between the original Cauchy problem and the formal limiting equation \eqref{eq:atw}.
It shows that the asymptotic traveling wave is not merely
a solution of the limiting profile equation, but is realized
as the long-time profile of the evolving front.
For fixed $K$, $\alpha$, and $u_T$, every initial datum
satisfying \eqref{condi:init} leads to the same limiting profile after recentering at the threshold location.
The proof of Theorem~\ref{thm:converge} is given in Section~\ref{sec:convergence}.

In the balanced case $\kappa=2u_T$, the convergence rate can be determined explicitly (see Remark~\ref{rem:balanced-rate}).
For nonzero speeds, no quantitative convergence rate is established here for general initial data satisfying \eqref{condi:init}.
An explicit $O(t^{-\alpha})$ estimate is nevertheless obtained in Proposition~\ref{prop:ct} for a separate class of specially chosen, continuous, generally nonsmooth initial data.

\bigskip
\section{Existence of global front solutions}\label{sec:global}

Throughout this section, we assume that Conditions~\eqref{condi:kernel}, \eqref{condi:bistable}, and \eqref{condi:init} hold. 
We first reduce the Cauchy problem to a nonlinear Volterra equation for the threshold location. 
We then establish the existence and uniqueness of a global solution to this equation.
The remaining subsections are devoted to the short- and long-time behavior of the threshold location and to the asymptotic linearity of its local increments.

\medskip
\subsection{Integral representation and nonlinear Volterra equation}\label{subsec:int-rep}

We begin with the standard variation-of-constants formula for a scalar time-fractional equation.
According to \cite{Diethelm}, the following result is known.

\begin{lemma}\label{lem:rep}
    Let $h\in C[0,+\infty)$ and $v_0\in\bR$. 
    The solution $v$ of 
    \be\label{pro:v}
        \begin{cases}
            \partial^{\alpha}_{t} v = - v + h(t) &(t>0) \\
            v(0) = v_0
        \end{cases}
    \ee
    is represented by
    \be\label{rep:v}
        v(t) = f_{\alpha}(t) v_0 + \int^{t}_{0} g_{\alpha}(t-s) h(s) ds \quad (t\ge 0),
    \ee
    where $f_{\alpha}(t) := E_{\alpha}(- t^{\alpha})$.

    Conversely, if $v\in C[0,+\infty)$ satisfies \eqref{rep:v}, then $v$ is a mild solution of \eqref{pro:v}.
\end{lemma}

Accordingly, throughout this paper, a solution of the Cauchy problem~\eqref{eq:main}--\eqref{init:main} is understood to be a mild solution unless otherwise stated.
Namely, $u\in C([0,\infty)\times\bR)$ is a mild solution if
\beaa
u(t,x) = f_{\alpha}(t) u_0(x) + \int^{t}_{0} g_{\alpha}(t-s)\int_{\bR} K(x-y) H(u(s,y)-u_T) dy ds
\eeaa
for all $t\ge 0$ and $x\in\bR$.

Suppose that $u$ is a front solution with threshold location $x^*(t)$. 
Then, every front solution satisfies
\be\label{sol:mild}
u(t,x) = f_{\alpha}(t) u_0(x) + \int^{t}_{0} g_{\alpha}(t-s)\cK(x-x^*(s)) ds.
\ee
For a continuous function $\xi\in C([0,\infty))$ satisfying $\xi(0)=0$, we therefore define 
\beaa
    U(t,x;\xi) := f_{\alpha}(t) u_0(x) + \int^t_{0} g_\alpha(t-s)\cK\bigl(x-\xi(s)\bigr)\,ds.  
\eeaa

We shall repeatedly use the following standard properties of the Mittag--Leffler functions.
\begin{lemma}\label{lem:ML}
    The following assertions hold.
    \begin{itemize}
        \item[(i)] $f_{\alpha}(t)$ is a continuous positive function on $[0,+\infty)$ satisfying
        \beaa
            \dlim_{t\to +\infty} t^{\alpha}  f_{\alpha}(t) = \dfrac{1}{\Gamma(1-\alpha)}.
        \eeaa
        Moreover, $f_{\alpha}(t)$ is a completely monotone function on $[0,+\infty)$;

        \item[(ii)] $g_{\alpha}(t)$ is a continuous positive function on $(0,+\infty)$ satisfying
        \be\label{ieq:galpha}
            g_{\alpha}(t) \le \dfrac{t^{\alpha-1}}{\Gamma(\alpha)}\qquad (t>0)
        \ee
        and
        \be\label{eq:g2f}
            \int^{t}_{0} g_{\alpha}(s) ds = 1 - f_{\alpha}(t).
        \ee
        It also holds that $g_{\alpha}(t)= - f'_{\alpha}(t)$;

        \item[(iii)] it holds that
        \beaa
            \frac{1}{\Gamma(1-\alpha)} \int^r_{0} (r-q)^{-\alpha}g_\alpha(q)\,dq = f_\alpha(r).
        \eeaa
    \end{itemize}
\end{lemma}

\medskip
The positivity and complete monotonicity of $f_{\alpha}(t)$ follow from \cite{Schneider}. 
The asymptotic formula can be found in \cite[Section~6]{HMS}. Identity \eqref{eq:g2f} follows directly from the definition of the Mittag--Leffler function.
Inequality~\eqref{ieq:galpha} is obtained from the monotonicity of $E_{\alpha,\alpha}$ and $E_{\alpha,\alpha}(0)=1/\Gamma(\alpha)$.
Identity~(iii) follows by taking Laplace transforms.

We next introduce the basic properties of $U$.
\begin{lemma}\label{lem:U}
    Suppose that $\xi\in C[0,+\infty)$ satisfies $\xi(0)=0$.
    $U(t,x;\xi)$ satisfies the following properties:
    \begin{itemize}
        \item[(i)]\ $U$ is continuous on $ [0,+\infty)\times\bR$;

        \item[(ii)]\ for any $t\ge0$, $U(t,x;\xi)$ has the limits
        \beaa
            \dlim_{x\to-\infty} U(t,x;\xi) = \kappa,\quad \dlim_{x\to+\infty} U(t,x;\xi) = 0;
        \eeaa

        \item[(iii)]\ for any $t\ge0$, $U(t,x;\xi)$ is strictly monotone decreasing with respect to $x$.
    \end{itemize}
\end{lemma}
\begin{proof}
    (i) Since $\cK$ is bounded and continuous and $\xi$ is continuous, the map 
    \beaa
       (s,x)\longmapsto\cK\bigl(x-\xi(s)\bigr)
    \eeaa   
    is continuous and bounded. 
    Moreover, $g_\alpha\in L^1(0,T)$ for every $T>0$.
    Standard continuity properties of Volterra convolutions therefore imply that $U(\cdot,\cdot;\xi)\in C([0,\infty)\times\bR)$.
    
    (ii) Using \eqref{condi:init}, the limits of $\cK$, and \eqref{eq:g2f}, we obtain
    \beaa
        \dlim_{x\to - \infty} U(t,x,\xi) &=& f_{\alpha}(t) \dlim_{x\to - \infty}  u_0(x) + \int^{t}_{0} g_{\alpha}(t-s) \left(  \dlim_{x\to - \infty} \cK(x-\xi(s)) \right) ds \\
        &=& \kappa f_{\alpha}(t) + \kappa \int^{t}_{0} g_{\alpha}(t-s) \, ds  =\kappa.
    \eeaa
    Similarly, we obtain the limit as $x\to+\infty$.

    (iii) Since $\cK'=-K$, differentiation with respect to $x$ gives
    \be\label{deri}
        \dfrac{\partial U}{\partial x} (t,x;\xi) =  f_{\alpha}(t) u'_0(x) - \int^{t}_{0}g_{\alpha}(t-s) K(x-\xi(s)) ds <0.
    \ee

    Thus, the desired assertions are shown.
\end{proof}

\medskip
The previous lemma reduces the construction of a front solution to the construction of its threshold location.
\begin{proposition}\label{prop:fixed}
    Suppose that there exists 
    $x^*\in C[0,\infty)$ satisfying $x^*(0)=0$ and
    \be\label{fixed}
        U\bigl(t,x^*(t);x^*\bigr)=u_T \qquad (t\geq0).
    \ee
    Then $u(t,x):=U(t,x;x^*)$ is a front solution of Cauchy problem~\eqref{eq:main}--\eqref{init:main}, and $x^*(t)$ is its threshold location. 
\end{proposition}
\begin{proof}
    By Lemma~\ref{lem:U}, the function $U(t,x;x^*)$ is strictly decreasing with respect to $x$.
    Hence, Equation~\eqref{fixed} implies that $u(t,x)= U(t,x;x^*)$ satisfies \eqref{sol:front}.

    Furthermore, $u(t,x)$ is a solution to the integral form of Equation~\eqref{eq:main} that satisfies \eqref{init:main}.
    Therefore, by Lemma~\ref{lem:rep}, we see that it is a mild solution to Equation~\eqref{eq:main} that satisfies \eqref{init:main}.
\end{proof}

\medskip
We prove the global existence and uniqueness of a solution to Equation~\eqref{fixed} in the next subsection.

\medskip
\subsection{Existence and uniqueness of the threshold location}

In this subsection, we prove that Equation~\eqref{fixed} admits a unique global solution. 
Fix $T>0$ and set 
\beaa
    \cX(T) := \left\{ \xi\in C([0,T])\,;\ \xi(0)=0 \right\}.
\eeaa
Note that for $\xi\in\cX(T)$, the extension
\beaa
    \xi(t) = \xi(T)\quad (t\ge T)
\eeaa
yields $\xi\in C[0,+\infty)$.
Thus, we can apply Lemma~\ref{lem:U} under this extension.

We first define an operator on $\cX(T)$ by assigning to each prescribed interface $\xi$ the corresponding threshold location of $U(\,\cdot\,,\,\cdot\,;\xi)$.
\begin{lemma}\label{lem:def-phi}
    For each $\xi\in \cX(T)$, there exists a unique $\eta\in \cX(T)$ such that 
    \be\label{eq:eta}
        U(t,\eta(t);\xi)=u_T \qquad (0\leq t\leq T).
    \ee
\end{lemma}
\begin{proof} 
    Fix $\xi\in \cX(T)$.
    For each $t\in[0,T]$, 
    Lemma \ref{lem:U} shows that 
    \beaa
        \lim_{x\to-\infty}U(t,x;\xi)=\kappa>u_T, \qquad \lim_{x\to+\infty}U(t,x;\xi)=0<u_T,
    \eeaa
    and that $U(t,x;\xi)$ is strictly decreasing with respect to $x$.
    Hence there exists a unique $\eta(t)\in\bR$ satisfying \eqref{eq:eta}.
    Since $U(0,x;\xi)=u_0(x)$ and $u_0(0)=u_T$, we have $\eta(0)=0$.
    
    It remains to prove the continuity of $\eta$. 
    Let $t_n\to t_0$ in $[0,T]$.
    For any $\eps>0$, we have 
    \beaa
        U(t_0,\eta(t_0)-\eps;\xi)>u_T, \qquad U(t_0,\eta(t_0)+\eps;\xi)<u_T. 
    \eeaa
    The continuity of $U$ implies that, for all sufficiently large $n$, 
    \beaa
        U(t_n,\eta(t_0)-\eps;\xi)>u_T, \qquad U(t_n,\eta(t_0)+\eps;\xi)<u_T. 
    \eeaa
    It follows that $|\eta(t_n) - \eta(t_0)|<\eps$.
    Since $\eps>0$ is arbitrary,
    \beaa
        \eta(t_n)\longrightarrow \eta(t_0).
    \eeaa
    Thus, we obtain $\eta\in \cX(T)$.
\end{proof}

\medskip
For $\xi\in \cX(T)$, let $\Phi(\xi)$ denote the unique function obtained in Lemma~\ref{lem:def-phi}, that is, 
\beaa 
U\bigl(t,\Phi(\xi)(t);\xi\bigr)=u_T\qquad (0\leq t\leq T). 
\eeaa
Then $\Phi:\cX(T)\to \cX(T)$ is well-defined, and a fixed point of $\Phi$ is a solution of Equation~\eqref{fixed} on $[0,T]$.

To construct a fixed point, we introduce a closed and bounded subset of $\cX(T)$.
Define 
\beaa
    \ell_\alpha(t):=\max\{t,t^\alpha\}, \qquad L_R(t):=R\ell_\alpha(t), 
\eeaa
and, for $R>0$, set 
\beaa
    \cX_R(T) := \left\{ \xi\in \cX(T)\,;\ |\xi(t)|\leq L_R(t) \text{ for all }t\in[0,T] \right\}.
\eeaa

We first establish two estimates that will be used away from $t=0$.
\begin{lemma}\label{lem:barrier_integrals}  
    Let $T>0$ and $\delta\in(0,T)$.
    Then 
    \be\label{lim:right}
    \lim_{R\to+\infty} \sup_{\delta\le t\le T} \int_0^t
    g_\alpha(t-s)\cK(L_R(t)-L_R(s))\,ds =0 
    \ee
    and 
    \be\label{lim:left}
    \lim_{R\to+\infty}\sup_{\delta\le t\le T}
    \int_0^t g_\alpha(t-s)  |\cK(-L_R(t)+L_R(s))-\kappa|\,ds = 0.
    \ee
\end{lemma}
\begin{proof}
    We only prove \eqref{lim:right}, since the proof of \eqref{lim:left} is analogous.

    Fix $\eps\in(0,\delta)$.
    Since $\ell_\alpha$ is continuous and strictly increasing,
    there exists $c_{\delta,\varepsilon}>0$ such that 
    \beaa
        \ell_\alpha(t)-\ell_\alpha(s) \geq c_{\delta,\eps} 
    \eeaa
    whenever $\delta\leq t\leq T$ and $0\leq s\leq t-\eps$. 
    Splitting the integral at $t-\varepsilon$,
    we obtain 
    \beaa
    && \int_0^t g_\alpha(t-s)\cK(L_R(t)-L_R(s))\,ds \\
    && = \int_0^{t-\eps} g_\alpha(t-s)\cK(L_R(t)-L_R(s))\,ds + \int_{t-\eps}^{t} g_\alpha(t-s)\cK(L_R(t)-L_R(s))\,ds \\
    && \leq \cK(Rc_{\delta,\eps}) + \kappa\int^{\eps}_{0} g_\alpha(r)\,dr.
    \eeaa
    First letting $R\to+\infty$, and then $\eps\to+ 0$, proves \eqref{lim:right}. 
\end{proof}

\medskip
We next prove that $\cX_R(T)$ is invariant under $\Phi$ when $R$ is sufficiently large.
\begin{lemma}\label{lem:Phi_invariant} 
    For every $T>0$, there exists $R>0$ such that 
    \beaa
        \Phi\bigl(\cX_R(T)\bigr)\subset \cX_R(T). 
    \eeaa
\end{lemma}
\begin{proof} 
    Fix \(T>0\). Since $\Phi(\xi)(0)=0$, it is enough to find $R>0$ such that 
    \be\label{ieq:lemma}
        U(t,-L_R(t);\xi)>u_T>U(t,L_R(t);\xi)
    \ee
    for every $t\in(0,T]$ and every $\xi\in \cX_R(T)$. 
    
    Set $C_{K} := \max\{u_T,\kappa-u_T\}$.
    Since $u_0'(0)<0$, there exist $a>0$ and $\rho>0$ such that 
    \beaa 
        u_0(y)\leq u_T-ay, \qquad u_0(-y)\geq u_T+ay \qquad (0\leq y\leq\rho). 
    \eeaa
    Moreover, there exist $\delta_0\in(0,\min\{1,T\})$ and $C_0>0$ such that 
    \beaa 
        f_\alpha(t)\geq\frac{1}{2}, \qquad 1-f_\alpha(t)\leq C_0t^\alpha \qquad (0<t\leq\delta_0),
    \eeaa
    since $f_{\alpha}(t)$ satisfies $1-f_\alpha(t)=O(t^\alpha)$ as $t\to +0$. 

    Let $0<t\leq\delta_0$, and suppose first that $L_R(t)\leq\rho$.
    Using \eqref{eq:g2f},
    we obtain 
    \beaa
        U(t,L_R(t);\xi)-u_T
        &=& f_\alpha(t) \{u_0(L_R(t))-u_T\} \\
        &&\quad + \int^{t}_{0} g_\alpha(t-s) \{\cK(L_R(t)-\xi(s))-u_T\}\,ds\\
        &\le& -aL_R(t) f_\alpha(t) + C_K\{1-f_\alpha(t)\}\\
        &\le& -\frac{a}{2}L_R(t)+C_KC_0 t^\alpha.
    \eeaa
    Since $t\leq\delta_0<1$, we have $L_R(t)=Rt^\alpha$.
    Therefore, 
    \beaa
        U(t,L_R(t);\xi)-u_T \leq \left( -\frac{aR}{2} + C_{K}C_0\right)t^\alpha. 
    \eeaa
    Similarly, we obtain
    \beaa
        U(t,-L_R(t);\xi)-u_T \geq \left( \frac{aR}{2} - C_{K}C_0 \right)t^\alpha. 
    \eeaa
    Consequently, both inequalities in \eqref{ieq:lemma} hold in this case whenever 
    \be\label{eq:R_small_time}
        R>\frac{2C_{K}C_0}{a}. 
    \ee 
    
    We next consider the case $0<t\leq\delta_0$ and $L_R(t)>\rho$.
    Set 
    \beaa
        m_\rho := \min\{ u_T-u_0(\rho), u_0(-\rho)-u_T \}>0.
    \eeaa
    Then 
    \beaa
      u_0(L_R(t))\leq u_T-m_\rho, \qquad u_0(-L_R(t))\geq u_T+m_\rho. 
    \eeaa
    It follows that 
    \beaa
        U(t,L_R(t);\xi)-u_T&\le& -m_\rho f_\alpha(t) +C_K (1-f_\alpha(t)), \\
        U(t,-L_R(t);\xi)-u_T&\ge& m_\rho f_\alpha(t) -C_K(1-f_\alpha(t))
    \eeaa
    Thus there exists $\delta_1\in(0,\delta_0]$, independent of $R$ and $\xi$, such that \eqref{ieq:lemma} holds for $0<t\leq\delta_1$, provided that \eqref{eq:R_small_time} is satisfied. 

    It remains to consider $t\in[\delta_1,T]$. 
    For $\xi\in \cX_R(T)$, the monotonicity of $\cK$ gives 
    \beaa
        U(t,L_R(t);\xi)
        &\le&
        f_\alpha(t)u_0(L_R(t))
        +\int^t_0 g_\alpha(t-s)\cK(L_R(t)-L_R(s))\,ds
    \eeaa
    and
    \beaa
        U(t,-L_R(t);\xi)
        &\ge& f_\alpha(t)u_0(-L_R(t))
        +\int^t_0 g_\alpha(t-s) \cK(-L_R(t)+L_R(s))\,ds.
    \eeaa
    As $R\to+\infty$, 
    \beaa
        u_0(L_R(t))\longrightarrow0, \qquad u_0(-L_R(t))\longrightarrow\kappa 
    \eeaa
    uniformly for $t\in[\delta_1,T]$.
    Lemma \ref{lem:barrier_integrals} therefore yields 
    \beaa
        \sup_{\delta_1\leq t\leq T} \sup_{\xi\in \cX_R(T)} U(t,L_R(t);\xi) \longrightarrow0,\qquad
        \inf_{\delta_1\leq t\leq T} \inf_{\xi\in \cX_R(T)} U(t,-L_R(t);\xi) \longrightarrow\kappa
    \eeaa
    as $R\to+\infty$.
    By \eqref{condi:bistable}, Inequalities \eqref{ieq:lemma} hold on $[\delta_1,T]$ for all sufficiently large $R$.
    Combining \eqref{eq:R_small_time} and the estimates on $[\delta_1,T]$, we obtain \eqref{ieq:lemma} for every $t\in(0,T]$.

    Since $U(t,x;\xi)$ is strictly decreasing with respect to $x$, the unique solution of $U(t,x;\xi)=u_T$ lies in $(-L_R(t),L_R(t))$.
    Hence $\Phi(\xi)\in \cX_R(T)$, which proves the assertion. 
\end{proof} 

\medskip
We now establish the continuity and compactness of $\Phi$.

\begin{lemma}\label{lem:Phi_compact} 
    Let $R>0$ be chosen as in Lemma~\ref{lem:Phi_invariant}.
    Then 
    \beaa
        \Phi:\cX_R(T)\longrightarrow \cX_R(T) 
    \eeaa
    is continuous and compact. 
\end{lemma} 
\begin{proof}
    Set $M:=L_R(T)$ and define 
    \beaa   
        m_{R,T} := f_\alpha(T) \min_{|x|\leq M}\{-u_0'(x)\}>0. 
    \eeaa
    Since $f_\alpha$ is non-increasing, it follows from \eqref{deri} that 
    \beaa 
        -\partial_x U(t,x;\xi) \geq m_{R,T} 
    \eeaa
    for every $t\in[0,T]$, $x\in[-M,M]$, and $\xi\in \cX_R(T)$.
    Consequently, 
    \be\label{eq:uniform_monotonicity_U}
        |U(t,x;\xi)-U(t,y;\xi)| \geq m_{R,T}|x-y| 
    \ee
    for all $x,y\in[-M,M]$. 
    
    We first prove the continuity of $\Phi$.
    Since $\cK'=-K$, the function $\cK$ is Lipschitz continuous with Lipschitz constant $\|K\|_{L^\infty(\bR)}$.
    Hence, for any $\xi_1,\xi_2\in \cX_R(T)$,
    \beaa
        |U(t,x;\xi_1)-U(t,x;\xi_2)|
        &\le&
        \int^t_0 g_\alpha(t-s)
        |\mathcal K(x-\xi_1(s))-\mathcal K(x-\xi_2(s))|\,ds \\
        &\le&
        \|K\|_{L^\infty(\bR)}
        \|\xi_1-\xi_2\|_{C[0,T]}
        \int^t_0 g_\alpha(t-s)\,ds \\
        &\le& \|K\|_{L^\infty(\bR)}
        \|\xi_1-\xi_2\|_{C[0,T]}
    \eeaa 
    Since $U(t,\Phi(\xi_j)(t);\xi_j)=u_T$ and $\Phi(\xi_j)(t)\in[-M,M]$, \eqref{eq:uniform_monotonicity_U} gives \beaa
        m_{R,T}|\Phi(\xi_1)(t)-\Phi(\xi_2)(t)| &\leq& \left| U(t,\Phi(\xi_1)(t);\xi_1) - U(t,\Phi(\xi_2)(t);\xi_1) \right| \\ 
        &=& \left| U(t,\Phi(\xi_2)(t);\xi_2) - U(t,\Phi(\xi_2)(t);\xi_1) \right| \\ &\leq& \|K\|_{L^\infty(\bR)} \|\xi_1-\xi_2\|_{C[0,T]}. 
    \eeaa
    Therefore, 
    \beaa
        \|\Phi(\xi_1)-\Phi(\xi_2)\|_{C[0,T]} \leq \frac{\|K\|_{L^\infty(\bR)}}{m_{R,T}} \|\xi_1-\xi_2\|_{C[0,T]}. 
    \eeaa
    Thus $\Phi$ is Lipschitz continuous. 
    
    We next prove compactness. 
    By Lemma~\ref{lem:Phi_invariant}, 
    \beaa
        |\Phi(\xi)(t)|\leq L_R(t)\leq M 
    \eeaa
    for every $\xi\in \cX_R(T)$ and $t\in[0,T]$.
    Hence $\Phi(\cX_R(T))$ is uniformly bounded. 
    It remains to prove equicontinuity. 
    Let $0\leq s\leq t\leq T$ and put $h:=t-s$. 
    Since $0\leq\cK\leq\kappa$, we have 
    \beaa
        |U(t,x;\xi)-U(s,x;\xi)| &\leq& \|u_0\|_{L^\infty(\bR)} |f_\alpha(t)-f_\alpha(s)| \\ &&\quad + \kappa \int^{s}_{0} \bigl| g_\alpha(t-r)-g_\alpha(s-r) \bigr|\,dr + \kappa \int^{t}_{s} g_\alpha(t-r)\,dr.
    \eeaa    
    Using the monotonicity of $g_\alpha$ and \eqref{eq:g2f}, we obtain 
    \beaa
        |f_\alpha(t)-f_\alpha(s)| = \int^{t}_{s} g_\alpha(r)\,dr \leq \int^h_0 g_\alpha(r)\,dr = 1-f_\alpha(h), 
    \eeaa
    and 
    \beaa
        \int^s_0 \bigl| g_\alpha(t-r)-g_\alpha(s-r) \bigr|\,dr = \int^s_0 \{g_\alpha(q)-g_\alpha(q+h)\}\,dq  
        \leq \int^h_0 g_\alpha(q)\,dq = 1-f_\alpha(h).
    \eeaa
    Moreover, 
    \beaa
        \int^t_{s} g_\alpha(t-r)\,dr = \int^h_0 g_\alpha(r)\,dr = 1-f_\alpha(h).
    \eeaa
    It follows that 
    \be\label{eq:uniform_time_modulus_U}
        \sup_{\xi\in \cX_R(T)} \sup_{x\in\bR} |U(t,x;\xi)-U(s,x;\xi)| \leq \omega(|t-s|), 
    \ee
    where $\omega(h) := \bigl( \|u_0\|_{L^\infty(\mathbb R)} + 2\kappa \bigr) \{1-f_\alpha(h)\}$.
    Since $f_\alpha(h)\to1$ as $h\to0$, we have $\omega(h)\to0$.

    Since $U(t,\Phi(\xi)(t);\xi) = U(s,\Phi(\xi)(s);\xi) = u_T$ holds, 
    \eqref{eq:uniform_monotonicity_U} and \eqref{eq:uniform_time_modulus_U} imply
    \beaa
        |\Phi(\xi)(t)-\Phi(\xi)(s)| &\leq& \dfrac{1}{m_{R,T}} |U(t,\Phi(\xi)(t);\xi)-U(t,\Phi(\xi)(s);\xi)| \\ 
        &=& \dfrac{1}{m_{R,T}} |U(s,\Phi(\xi)(s);\xi)-U(t,\Phi(\xi)(s);\xi)| \\ 
        &\leq& \dfrac{1}{m_{R,T}} \omega(|t-s|)
    \eeaa
     uniformly with respect to $\xi\in \cX_R(T)$.
     Thus $\Phi(\cX_R(T))$ is equicontinuous. 

     The Ascoli--Alzel\'{a} theorem now shows that $\Phi(\cX_R(T))$ is relatively compact in $C[0,T]$.
     Hence $\Phi$ is compact.
\end{proof}

\medskip
We are now in a position to prove the existence and uniqueness of the threshold location.
\begin{proposition}\label{prop:exi-uni}
    There exists a unique function $x^*\in C[0,\infty)$ satisfying $x^*(0)=0$ and Equation~\eqref{fixed}.  
\end{proposition}
\begin{proof}
    Fix $T>0$.
    By Lemma~\ref{lem:Phi_invariant}, there exists $R>0$ such that 
    \beaa
        \Phi:\cX_R(T)\longrightarrow \cX_R(T).
    \eeaa
    The set $\cX_R(T)$ is nonempty, closed, bounded, and convex in $C[0,T]$.
    By Lemma~\ref{lem:Phi_compact}, the map $\Phi$ is continuous and compact. 
    Schauder's fixed-point theorem therefore yields a function $x^*\in \cX_R(T)$ such that 
    \beaa
        x^*=\Phi(x^*). 
    \eeaa
    Thus \eqref{fixed} has at least one solution on $[0,T]$.
    
    We next prove uniqueness. 
    Let $x_1^*,x_2^*\in \cX(T)$ be two solutions on $[0,T]$.
    Since both functions are bounded on $[0,T]$,
    there exists $M>0$ such that 
    \beaa
        |x_j^*(t)|\leq M \qquad (0\leq t\leq T, \ j=1,2). 
    \eeaa
    Set 
    \beaa
     c_{M,T}:= f_\alpha(T) \min_{|x|\leq M}\{-u_0'(x)\}>0. 
    \eeaa
    Using $U\bigl(t,x_j^*(t);x_j^*\bigr) = u_T$ for $j=1,2$,
    we write
   \beaa
        U(t,x^*_1(t);x^*_1) - U(t,x^*_2(t);x^*_1)
        = U(t,x^*_2(t);x^*_2) - U(t,x^*_2(t);x^*_1) 
    \eeaa
    The uniform strict monotonicity in $x$ gives 
    \beaa 
        c_{M,T}|x_1^*(t)-x_2^*(t)| \leq{}& \left| U\bigl(t,x_2^*(t);x_2^*\bigr) - U\bigl(t,x_2^*(t);x_1^*\bigr) \right|.
    \eeaa
    Since $\cK$ is Lipschitz continuous, 
    \beaa
        |x_1^*(t)-x_2^*(t)| \leq C_{M,T} \int^t_0 g_\alpha(t-s) |x_1^*(s)-x_2^*(s)|\,ds, 
    \eeaa
    where $C_{M,T} := \|K\|_{L^\infty(\bR)}/c_{M,T}$.
    Using \eqref{ieq:galpha}, we obtain
    \beaa
        |x_1^*(t)-x_2^*(t)| \leq \dfrac{C_{M,T}}{\Gamma(\alpha)} \int^t_0 (t-s)^{\alpha-1} |x_1^*(s)-x_2^*(s)|\,ds
    \eeaa
    The fractional Gr\"{o}nwall inequality \cite[Lemma~7.1.1]{Henry} yields $x_1^*(t) = x_2^*(t)$ for all $t\in[0,T]$.
    Thus the solution of \eqref{fixed} is unique on $[0,T]$.
    
    Therefore, $\Phi:\cX(T)\to \cX(T)$ has a unique fixed point for every $T>0$.
    This proves the existence and uniqueness of the threshold location $x^*\in C[0,+\infty)$ satisfying $x^*(0)=0$ and \eqref{fixed}.
    
    The proof is complete.
\end{proof}

\medskip
By Propositions~\ref{prop:fixed} and~\ref{prop:exi-uni}, the function 
\beaa
    u(t,x):=U(t,x;x^*) 
\eeaa
is the unique global front solution to Equation~\eqref{eq:main} with \eqref{init:main}.
Together with Lemma~\ref{lem:U}, this proves the existence and uniqueness part of Theorem~\ref{thm:global}.

\medskip
\subsection{Short- and long-time asymptotics of the threshold location}

In this subsection, we investigate the behavior of the threshold location near the initial time and as $t\to+\infty$. 

We first derive its short-time asymptotic expansion.
\begin{proposition}\label{prop:asym-0}
    The threshold location $x^*(t)$ satisfies \eqref{xs:short}.
\end{proposition}
\begin{proof}
    Note that the construction of $x^*(t)$ in the previous subsection yields $x^*(t)=O(t^\alpha)$ as $t\to+0$.
    By the definition of the Mittag--Leffler function, we deduce 
    \beaa 
        f_\alpha(t) = 1-\frac{t^\alpha}{\Gamma(1+\alpha)} + o(t^\alpha).
    \eeaa
    Since $u_0(0)=u_T$ and $u_0$ is differentiable at $x=0$, we have
    \beaa
        u_0\bigl(x^*(t)\bigr) = u_T+u_0'(0)x^*(t)+o(t^\alpha). 
    \eeaa
    Consequently,
    \be\label{lim0-1}
        f_\alpha(t)u_0(x^*(t))= u_T+u'_0(0)x^*(t)
        -\frac{u_T}{\Gamma(1+\alpha)}t^\alpha+o(t^\alpha)
    \ee
    as $t\to+0$.

    For $0\leq r\leq t$, we have
    \beaa
        \left| x^*(t)-x^*(t-r) \right| = O(t^\alpha) 
    \eeaa
    uniformly in $r$.
    Since $\cK$ is continuous and $\cK(0)=\kappa/2$, it follows that
    \beaa
        \sup_{0\le r\le t}\left|\cK(x^*(t)-x^*(t-r))-\mathcal K(0)\right|\to 0.
    \eeaa
    Moreover, 
    \beaa
        \int^t_0 g_\alpha(r)\,dr = 1-f_\alpha(t) = \frac{t^\alpha}{\Gamma(1+\alpha)} + o(t^\alpha). 
    \eeaa
    Therefore, 
    \be\label{lim0-2}
    \int^t_0 g_\alpha(r)\cK(x^*(t)-x^*(t-r))\,dr
    =\frac{\kappa}{2\Gamma(1+\alpha)}t^\alpha+o(t^\alpha).
    \ee
    as $t\to+0$.

    Since $x^*(t)$ satisfies Equation~\eqref{fixed},
    we obtain
    \beaa
        u_0'(0)x^*(t) +\frac{\kappa-2u_T}{2\Gamma(1+\alpha)}t^\alpha+o(t^\alpha)=0
    \eeaa
    from \eqref{lim0-1} and \eqref{lim0-2}.
    Note that $u_0'(0)<0$.
    This proves \eqref{xs:short}.
\end{proof}

\medskip
We next determine the asymptotic propagation speed of the threshold location.
For later use, we first introduce some basic properties of $G_\alpha$, which are easy to verify from Lemma~\ref{lem:ML}.
\begin{lemma}\label{lem:properties-G}
    The function $G_\alpha$ is continuous and strictly decreasing on $\bR$. 
    Moreover,
    \beaa
    \lim_{c\to-\infty}G_\alpha(c)=\kappa,
    \qquad
    G_\alpha(0)=\frac{\kappa}{2},
    \qquad
    \lim_{c\to+\infty}G_\alpha(c)=0.
    \eeaa
    For $c\neq0$, it also admits the representation
    \beaa
    G_\alpha(c)=
    \frac{\kappa}{2}
    -
    c\int_0^\infty f_\alpha(r)K(cr)\,dr.
    \eeaa
\end{lemma}

\medskip
Let us consider the asymptotic behavior of the threshold location.

\begin{proposition}\label{prop:asym-inf}
    The threshold location $x^*(t)$ satisfies \eqref{xs:long}.
\end{proposition}
\begin{proof}
    Let $c^*$ be the unique solution of Equation~\eqref{eq:speed}.
    We first prove that 
    \be\label{speed-limsup}
        \limsup_{t\to+\infty} \dfrac{x^*(t)}{t} \le c^*
    \ee

    Fix $q>c^*$. 
    Set $\gamma := u_T-G_\alpha(q) > 0$. 
    Choose $T_0>0$ so large that $\kappa f_\alpha(t) < \gamma/2$ holds for any $t\geq T_0$.
    By the continuity of $x^*$, we may choose $a>0$ such that 
    $x^*(t)<qt+a$ for all $t\in[0,T_0]$.
    We claim that 
    \be\label{eq:upper-barrier}
        x^*(t)<qt+a\quad (\forall t\ge 0).
    \ee

    Suppose otherwise. 
    Then there exists a first contact time $t_0>T_0$ such that 
    \beaa
        x^*(t)<qt+a \quad (0\leq t<t_0), \qquad x^*(t_0)=qt_0+a. 
    \eeaa
    This implies that $x^*(t_0)-x^*(t_0-r) > qr$ for every $r\in(0,t_0]$. 
    Since $\cK$ is non-increasing and $0<u_0<\kappa$, we obtain 
    \beaa 
        u_T &=& f_\alpha(t_0)u_0\bigl(x^*(t_0)\bigr) + \int^{t_0}_0 g_\alpha(r) \cK\bigl(x^*(t_0)-x^*(t_0-r)\bigr)\,dr \\ 
        &\leq& \kappa f_\alpha(t_0) + \int^{t_0}_{0} g_\alpha(r)\cK(qr)\,dr  
        < \frac{\gamma}{2} + G_\alpha(q) = u_T-\frac{\gamma}{2},
    \eeaa
    which is a contradiction. 
    Hence \eqref{eq:upper-barrier} holds.

    Dividing \eqref{eq:upper-barrier} by $t$ and letting $t\to+\infty$, we obtain 
    \beaa
        \limsup_{t\to+\infty} \frac{x^*(t)}{t} \leq q. 
    \eeaa
    Since $q>c^*$ is arbitrary, \eqref{speed-limsup} follows. 
    
    Similarly, we obtain
    \be\label{speed:liminf}
        \liminf_{t\to+\infty} \dfrac{x^*(t)}{t} \ge c^*
    \ee
    Combining \eqref{speed-limsup} and \eqref{speed:liminf}, we obtain \eqref{xs:long}.
\end{proof}

\medskip
Now, we give the proof of Theorem~\ref{thm:global}.
\begin{proof}[Proof of Theorem~2.2]
Proposition~\ref{prop:exi-uni} and Lemma~\ref{lem:U} establish the existence, uniqueness, and
both-side limits of the global front solution. 
The short- and long-time
asymptotics of its threshold location follow from Propositions~\ref{prop:asym-0} and
\ref{prop:asym-inf}, respectively. 
This completes the proof.
\end{proof}

\medskip
\subsection{Asymptotic linearity of threshold increments}

In this subsection, we establish the following property of the threshold location that will be used in the proof of convergence to the asymptotic traveling wave. 
\begin{proposition}\label{prop:dif}
    Suppose that $\kappa \neq 2u_T$.
    For every $L>0$, we have
    \beaa
        \dlim_{t\to+\infty}\sup_{0\le r\le L}|x^*(t)-x^*(t-r) - c^*r| = 0.
    \eeaa
\end{proposition}

\medskip
Throughout this subsection, we treat only the case $\kappa>2u_T$, since the case $\kappa<2u_T$ follows along the same lines.
We first establish the sign and monotonicity of the threshold location.

\begin{lemma}\label{lem:positive}
    Assume that $\kappa\neq2u_T$. 
    Then $x^*(t)$ has the same sign as $c^*$ for every $t>0$.
    Moreover, $x^*$ is strictly increasing when $c^*>0$, and strictly decreasing when $c^*<0$.
\end{lemma}
\begin{proof} 
    By \eqref{xs:short}, $x^*(t)>0$ holds for all sufficiently small $t>0$.
    Suppose that $x^*$ vanishes at some positive time, and let $t_0>0$ be the first such time. 
    Then 
    \beaa
        x^*(t)>0 \quad (0<t<t_0), \qquad x^*(t_0)=0. 
    \eeaa
    Since $\cK$ is non-increasing, 
    \beaa
        \cK\bigl(-x^*(s)\bigr) \geq \mathcal K(0) = \frac{\kappa}{2} \qquad (0<s<t_0). 
    \eeaa
    Thus, we obtain
    \beaa
    u_T \geq f_\alpha(t_0)u_T + \frac{\kappa}{2} \{1-f_\alpha(t_0)\}.
    \eeaa
    from \eqref{fixed}.
    Since $f_\alpha(t_0)<1$, this implies $\kappa\leq 2 u_T$, which is a contradiction. 
    Hence $x^*(t)>0$ for every $t>0$. 
    
    We next prove that $x^*$ is strictly increasing.
    Fix $m>0$.
    By Proposition~\ref{prop:asym-inf}, $x^*(t)\to+\infty$ as $t\to+\infty$.
    Hence, by continuity and $x^*(0)=0$, the first hitting time
    \beaa
        t_m:=\min\{t\ge0:x^*(t)=m\}
    \eeaa
    is well-defined and positive. Moreover,
    \beaa
        x^*(t_m)=m, \qquad x^*(s)<m \quad (0\le s<t_m).
    \eeaa
    Set $d(t):=(m-x^*(t))_{+}$, where $(r)_{+}:= \max\{r,0\}$.
    We show that $d(t)\equiv 0$ for any $t\ge t_m$.

    Using the integral representation of $u$, we have
    \beaa
        u(t,m)-\dfrac{\kappa}{2} = f_\alpha(t)\left(u_0(m)-\dfrac{\kappa}{2}\right)  + \int^{t}_0 g_\alpha(t-s) \left[ \cK\bigl(m-x^*(s)\bigr)-\dfrac{\kappa}{2}\right]\,ds.
    \eeaa
    Since $u_0(m)<u_T<\kappa/2$ and
    \beaa
        \cK\bigl(m-x^*(s)\bigr)-\frac{\kappa}{2} \le0 \qquad (0\le s\le t_m),
    \eeaa
    subtracting the identity at $t=t_m$ and using the monotonicity of $g_\alpha$, we obtain
    \be\label{ieq:mono}
        u(t,m)-u_T\ge \left(\dfrac{\kappa}{2}-u_0(m)\right) \bigl[f_\alpha(t_m)-f_\alpha(t)\bigr] + \int^{t}_{t_m} g_\alpha(t-s) \left[ \cK\bigl(m-x^*(s)\bigr)-\frac{\kappa}{2} \right]\,ds
    \ee
    for every $t\ge t_m$.
    Moreover, since $\cK'=-K$ and $\cK(0)=\kappa/2$,
    \beaa
        \cK\bigl(m-x^*(s)\bigr)-\dfrac{\kappa}{2} \ge -\|K\|_{L^\infty(\bR)}d(s).
    \eeaa
    Fix $T>t_m$ and set
    \beaa
        c_T:= f_\alpha(T)\min_{0\le x\le m}\{-u_0'(x)\}>0.
    \eeaa
    Since $x^*(t)>0$, \eqref{deri} gives
    \beaa
        c_T d(t) \le \bigl(u_T-u(t,m)\bigr)_+ \qquad (t\in[t_m,T]).
    \eeaa
    Therefore, \eqref{ieq:mono} and the monotonicity of $f_\alpha$ imply
    \beaa
        d(t) \le \dfrac{\|K\|_{L^\infty(\bR)}}{c_T} \int^{t}_{t_m} g_\alpha(t-s)d(s)\,ds \qquad (t\in[t_m,T]).
    \eeaa
    Using \eqref{ieq:galpha},
    the fractional Gr\"{o}nwall inequality yields $d(t)=0$ for all $t\in[t_m,T]$.
    Since $T>t_m$ is arbitrary, we conclude that
    \beaa
        x^*(t)\ge m \qquad (t\ge t_m).
    \eeaa

    The integral term in \eqref{ieq:mono} is therefore nonnegative. 
    Since $f_\alpha$ is strictly decreasing,
    \beaa
        u(t,m)-u_T \ge \left(\frac{\kappa}{2}-u_0(m)\right) \bigl[f_\alpha(t_m)-f_\alpha(t)\bigr] >0 \qquad (t>t_m).
    \eeaa
    By the strict monotonicity of $u(t,\cdot)$, this implies
    \beaa
        x^*(t)>m \qquad (t>t_m).
    \eeaa

    Finally, let $0<s<t$ and set $m=x^*(s)>0$.
    Then $t_m\le s<t$, and hence
    \beaa
        x^*(t)>m=x^*(s).
    \eeaa
    Together with $x^*(0)=0<x^*(t)$ for $t>0$, this proves that $x^*$ is strictly increasing.
\end{proof}

\medskip
Next, we consider the limiting profile of $x^*(t)-x^*(t-r)$.
Extend the threshold location to the whole line by setting
\beaa
    \overline{x}^*(t):=
    \begin{cases}
        x^*(t) &(t\ge 0), \\
        0 &(t<0).
    \end{cases}
\eeaa
This extension is continuous and non-decreasing.
Since $x^*(0)=0$, the threshold equation~\eqref{fixed} can be rewritten as
\be\label{conv:eq-extended-threshold}
    u_T = b(t) + \int^{+\infty}_{0}  g_{\alpha}(r)  \cK(\overline{x}^*(t) - \overline{x}^*(t-r))\,dr,
\ee
where 
\beaa
    b(t) := 
    \begin{cases}
        f_{\alpha}(t)\bigl[ u_0(x^*(t)) - \cK(x^*(t)) \bigr] &(t\ge 0), \\
        u_T - \kappa/2 &(t<0).
    \end{cases}
\eeaa
Since $0\le u_0,\ \cK\le \kappa$, we have
\beaa
    |b(t)| \le \kappa f_{\alpha}(t)\longrightarrow 0\quad (t\to+\infty).
\eeaa

We first establish a compactness result that applies both to time translations of \eqref{conv:eq-extended-threshold} and to translations of solutions of its limiting equation.

\begin{lemma}\label{lem:compactness}
    Let $q_n\in C(\bR)$ be non-decreasing functions with $q_n(0)=0$.
    Suppose that $b_n:\bR\to\bR$ satisfy $b_n\to0$ locally uniformly on $\bR$ and that
    \be\label{conv:eq-approximate}
      u_T=b_n(s)
      +\int^{+\infty}_{0} g_\alpha(r)
        \cK\bigl(q_n(s)-q_n(s-r)\bigr)\,dr\qquad (s\in\bR).
    \ee
    Then there exist a subsequence, still denoted by $\{q_n\}$,
    and a continuous non-decreasing function $q:\bR\to\bR$ such that
    \beaa
        q_n\longrightarrow q\qquad\text{in }C_{\mathrm{loc}}(\bR).
    \eeaa
    Furthermore, $q(0)=0$ and
    \be\label{conv:eq-limit}   u_T=\int^{+\infty}_{0}g_\alpha(r)
        \cK\bigl(q(s)-q(s-r)\bigr)\,dr \qquad (s\in\bR).
    \ee
\end{lemma}
\begin{proof}
    \textbf{Step 1: Local bounds.}
    Choose $h>0$ so small that
    \beaa
      \dfrac{\kappa}{2}\int^h_0 g_\alpha(r)\,dr<\frac{u_T}{4},
    \eeaa
    and then choose $R>0$ so large that $\cK(R)<u_T/2$.
    Whenever $|b_n(s)|\le u_T/4$, the monotonicity of $q_n$ and $\cK$ gives
    \beaa
        u_T &\le&\dfrac{u_T}{4} +\dfrac{\kappa}{2}\int^h_0 g_\alpha(r)\,dr +\cK\bigl(q_n(s)-q_n(s-h)\bigr) \int^{+\infty}_h g_\alpha(r)\,dr\\
        &\le&\dfrac{u_T}{2} +\cK\bigl(q_n(s)-q_n(s-h)\bigr).
    \eeaa
    It follows that
    \be\label{conv:eq-step-bound}
        0\le q_n(s)-q_n(s-h)\le R.
    \ee
    Here $h$ and $R$ are independent of $n$ and $s$.

    Fix $L>0$ and set $N:=\lceil L/h\rceil$.
    For all sufficiently large $n$, the bound $|b_n(s)|\le u_T/4$ holds on $[-Nh,Nh]$. Summing \eqref{conv:eq-step-bound} over consecutive intervals of length $h$, and using $q_n(0)=0$ and monotonicity, we obtain
    \beaa
        \sup_{|s|\le L}|q_n(s)|\le NR.
    \eeaa
    Thus $\{q_n\}$ is locally uniformly bounded.

    \smallskip
    \noindent
    \textbf{Step 2: Passage to the limit.}
    By Helly's selection theorem and a diagonal argument, there exist a subsequence and a non-decreasing function $q:\bR\to\bR$ such that $q_n(s)\to q(s)$ at every continuity point of $q$.
    Fix such a point $s$.
    A monotone function has at most countably many discontinuities;
    hence, for almost every $r>0$,
    \beaa
        q_n(s)-q_n(s-r)\longrightarrow q(s)-q(s-r).
    \eeaa
    Moreover,
    \beaa
        0\le g_\alpha(r)\cK\bigl(q_n(s)-q_n(s-r)\bigr)\le\dfrac{\kappa}{2}g_\alpha(r),
    \eeaa
    and the right-hand side is integrable on $(0,+\infty)$.
    The dominated convergence theorem applied to \eqref{conv:eq-approximate} therefore yields \eqref{conv:eq-limit} at every continuity point $s$ of $q$.

    \smallskip
    \noindent
    \textbf{Step 3: Exclusion of jumps.}
    Suppose that $q$ has a jump at some $s_0\in\bR$, that is,
    \beaa
        q(s^+_0)-q(s^-_0)>0.
    \eeaa
    Approach $s_0$ from the left and from the right through continuity points of $q$.
    For almost every $r>0$, the point $s_0-r$ is also a continuity point of $q$.
    Passing to the limit in \eqref{conv:eq-limit} along these two sequences gives
    \be\label{conv:eq-jump-identities}
        u_T=\int^{+\infty}_{0} g_\alpha(r)
            \cK\bigl(q(s^-_0)-q(s_0-r)\bigr)\,dr =\int^{+\infty}_{0} g_\alpha(r)
            \cK\bigl(q(s^+_0)-q(s_0-r)\bigr)\,dr.
    \ee
    Set $A(r):=q(s^-_0)-q(s_0-r)$.
    Then $A(r)\ge0$ and $A(r)\to0$ as $r\to+0$.
    Note that Condition~\eqref{condi:kernel} yields constants $\eta>0$ and $k_0>0$ such that
    \be\label{conv:eq-local-positivity}
        K(y)\ge k_0\qquad(0\le y\le\eta).
    \ee
    Consequently, there exists $\delta>0$ such that
    $0\le A(r)\le\eta/2$ for $0<r<\delta$.
    By \eqref{conv:eq-local-positivity},
    \beaa
        \cK(A(r))-\cK(A(r)+q(s^+_0)-q(s^-_0)) &=&\int_{A(r)}^{A(r)+q(s^+_0)-q(s^-_0)}K(y)\,dy\\
         &\ge& k_0\min\{q(s^+_0)-q(s^-_0),\eta/2\}>0
        \qquad(0<r<\delta).
    \eeaa
    The same difference is nonnegative for all $r>0$.
    Subtracting the two integral expressions in \eqref{conv:eq-jump-identities},
    we obtain
    \beaa
      0\ge k_0\min\{q(s^+_0)-q(s^-_0),\eta/2\} \int^\delta_0 g_\alpha(r)\,dr>0,
    \eeaa
    a contradiction.
    Thus $q$ is continuous.

    The monotonicity of $q_n$ and the continuity of $q$ now imply local uniform convergence.
    In particular, $q(0)=0$.
    Since every point is now a continuity point, \eqref{conv:eq-limit} holds for all $s\in\bR$.
\end{proof}

\medskip
We next show that every continuous non-decreasing solution of
the limiting threshold equation is linear.

\begin{lemma}\label{lem:entire}
   Let $q:\bR\to\bR$ be a continuous non-decreasing solution of \eqref{conv:eq-limit}.
    Then
    \beaa
        q(s)=q(0)+c^*s\qquad(s\in\bR),
    \eeaa
    where $c^*>0$ is the unique solution of Equation~\eqref{eq:speed}.
\end{lemma}
\begin{proof}
    Fix $\theta>0$ and set
    \beaa
        D_\theta(s):=q(s+\theta)-q(s).
    \eeaa
    The argument leading to \eqref{conv:eq-step-bound}, with zero error, applies at every $s\in\bR$.
    It follows that
    \beaa
        0\le D_\theta(s)\le R\lceil\theta/h\rceil \qquad(s\in\bR).
    \eeaa
    In particular, the quantities
    \beaa
        M_\theta:=\sup_{s\in\bR}D_\theta(s),\qquad m_\theta:=\inf_{s\in\bR}D_\theta(s)
    \eeaa
    are finite.

    \smallskip
    \noindent
    \textbf{Step 1: Realization of an extremal increment.}
    Let $\mu$ be either $M_\theta$ or $m_\theta$.
    Choose a sequence $\{s_n\}\subset\bR$ such that   $D_\theta(s_n)\to \mu$, and define
    \beaa
        Q_n(s):=q(s+s_n)-q(s_n).
    \eeaa
    Equation~\eqref{conv:eq-limit} is invariant under time translations and addition of constants, so $Q_n$ also satisfies this equation and $Q_n(0)=0$.
    Lemma~\ref{lem:compactness}, with $b_n\equiv0$, provides a subsequence such that
    \beaa
      Q_n\longrightarrow Q\qquad\text{in }C_{loc}(\bR),
    \eeaa
    where $Q$ is a continuous non-decreasing solution of \eqref{conv:eq-limit} with $Q(0)=0$.
    Set
    \beaa
        \Delta(s):=Q(s+\theta)-Q(s).
    \eeaa
    Then $\Delta(0)=\mu$ and
    \beaa
    \begin{cases}
        \Delta(s)\le \mu&\text{if }\mu=M_\theta,\\
        \Delta(s)\ge \mu&\text{if }\mu=m_\theta
    \end{cases}
    \qquad(s\in\bR).
    \eeaa

    \smallskip
    \noindent
    \textbf{Step 2: Backward propagation of the extremal value.}
    Suppose that $\Delta(s_0)=\mu$ for some $s_0\in\bR$, and write
    \beaa
        A(r):=Q(s_0)-Q(s_0-r),\qquad B(r):=Q(s_0+\theta)-Q(s_0+\theta-r).
    \eeaa
    Then
    \beaa
        B(r)-A(r)=\mu-\Delta(s_0-r).
    \eeaa
    Subtracting \eqref{conv:eq-limit} at $s=s_0$ from the same equation at $s=s_0+\theta$ yields
    \beaa
        0=\int^{+\infty}_{0} g_\alpha(r)
       \bigl[\cK(B(r))-\cK(A(r))\bigr]\,dr.
    \eeaa
    If $\mu=M_\theta$, then $B(r)\ge A(r)$, so the integrand is nonpositive.
    If $\mu=m_\theta$, then $B(r)\le A(r)$, so the integrand is nonnegative.
    Since $g_\alpha(r)>0$ for $r>0$ and $\cK(B(r))-\cK(A(r))$ is continuous, in either case we obtain
    \beaa
      \cK(B(r))=\cK(A(r))\qquad(r>0).
    \eeaa
    By the continuity and monotonicity of $Q$, both $A(r)$ and $B(r)$ are nonnegative and tend to zero as $r\to+0$.
    They therefore belong to $[0,\eta]$ for all sufficiently small $r>0$.
    The strict monotonicity of $\cK$ on this interval implies $A(r)=B(r)$ for such $r$.
    Consequently, there exists $\delta(s_0)>0$ such that
    \be\label{conv:eq-local-propagation}
        \Delta(s_0-r)=\mu \qquad(0\le r\le\delta(s_0)).
    \ee  

    To extend this conclusion to the whole negative half-line, set
    \beaa
        \ell:=\inf\{a\le0:\ \Delta(s)=\mu\text{ for every }s\in[a,0]\}.
    \eeaa
    Because $\Delta(0)=\mu$, \eqref{conv:eq-local-propagation} shows that the set in this definition contains a negative number.
    If $\ell> -\infty$, continuity gives $\Delta(s)=\mu$ on $[\ell,0]$.
    Applying \eqref{conv:eq-local-propagation} at $s_0=\ell$ extends this interval to the left, contradicting the definition of $\ell$.
    Hence $\ell=-\infty$, and
    \be\label{conv:eq-half-line-increment}
        Q(s+\theta)-Q(s)=\mu\qquad(s\le0).
    \ee
    In particular, no uniform lower bound on $\delta(s_0)$ is needed.

    \smallskip
    \noindent
    \textbf{Step 3: Identification of the extremal value.}
    Set
    \beaa
        d_\theta:=\frac{\mu}{\theta}\ge0, \qquad p(s):=Q(s)-d_\theta s.
    \eeaa
    By \eqref{conv:eq-half-line-increment}, $p$ is $\theta$-periodic on $(-\infty,0]$.
    Let $s_+,s_-\in[-\theta,0]$ be a maximum point and a minimum point of $p$, respectively.
    Periodicity gives, for every $r\ge0$,
    \beaa
        Q(s_+)-Q(s_+-r)\ge d_\theta r,\qquad Q(s_-)-Q(s_--r)\le d_\theta r.
    \eeaa
    Using \eqref{conv:eq-limit} at $s=s_+$ and $s=s_-$, together with the monotonicity of $\cK$, we obtain
    \beaa
        u_T\le\int^{+\infty}_{0} g_\alpha(r)\cK(d_\theta r)\,dr
       =G_\alpha(d_\theta)
       \le u_T.
    \eeaa
    Thus $G_\alpha(d_\theta)=u_T$.
    By Lemma~\ref{lem:properties-G} and Equation~\eqref{eq:speed}, $d_\theta=c^*$, and hence
    $\mu=c^*\theta$.
    Since the argument applies to both $\mu=M_\theta$ and $\mu=m_\theta$,
    \beaa
        M_\theta=m_\theta=c^*\theta.
    \eeaa
    Therefore,
    \beaa
        q(s+\theta)-q(s)=c^*\theta\qquad(s\in\bR).
    \eeaa
    As $\theta>0$ is arbitrary, this proves the assertion.
\end{proof}
    
\medskip
We are now in a position to prove Proposition~\ref{prop:dif}.
\begin{proof}[Proof of Proposition~\ref{prop:dif}]
    Suppose that $\kappa>2u_T$.
    Let $\{t_n\}\subset[0,+\infty)$ be any sequence tending to $+\infty$, and set
    \beaa
        q_n(s):=\overline{x}^{\,*}(t_n+s)-\overline{x}^{\,*}(t_n) \qquad (s\in\bR).
    \eeaa
    The functions $q_n$ are continuous and non-decreasing, and $q_n(0)=0$.
    By \eqref{conv:eq-extended-threshold},
    \beaa
        u_T=b(t_n+s) +\int^{+\infty}_{0} g_\alpha(r) \cK\bigl(q_n(s)-q_n(s-r)\bigr)\,dr.
    \eeaa
    Note that $b(t_n+\cdot)\to0$ locally uniformly on $\bR$.
    By Lemma~\ref{lem:compactness}, we may extract a subsequence that converges locally uniformly to a function $q$.
    The limit $q$ is continuous and non-decreasing, satisfies \eqref{conv:eq-limit}, and has $q(0)=0$.
    Lemma~\ref{lem:entire} then gives $q(s)=c^*s$.

    Since every sequence $t_n\to+\infty$ admits a subsequence with this same limit, for every $L>0$ we have
    \be\label{conv:eq-translate-convergence}
        \lim_{t\to+\infty}\sup_{|s|\le L} \bigl|\overline{x}^{\,*}(t+s)-\overline{x}^{\,*}(t)-c^*s\bigr|=0.
    \ee
    For $t\ge L$, the extension $\overline{x}^{\,*}$ agrees with $x^*$ at all arguments in \eqref{conv:eq-translate-convergence}.
    Taking $s=-r$ therefore yields the assertion of Proposition~\ref{prop:dif}.
\end{proof}

\medskip
We have proved that the local increments of the threshold location are asymptotically linear with slope $c^*$.
This property will play a key role in the convergence analysis in Section~\ref{sec:convergence}.

\bigskip
\section{Construction and uniqueness of asymptotic traveling waves}\label{sec:atw}

\medskip
In this section, we construct the profile of asymptotic traveling waves. 
We first establish the basic properties of the profile $\phi(\,\cdot\,;c)$ defined by \eqref{phi}. 
We then show that uniqueness may fail
in the zero-speed case. 
Finally, a strong maximum principle is used to prove uniqueness of both the speed and the profile when the speed is nonzero.

We begin with the linear equation satisfied by the profile.

\begin{lemma}\label{lem:phi}
    For any $c\in\bR$, the function $\phi(\,\cdot\,;c)$ defined by \eqref{phi} belongs to $BC^1(\bR)$ and is non-increasing. 
    Moreover, it satisfies
    \be\label{sol:eq}
        D_c^\alpha\phi(z;c)+\phi(z;c)=\cK(z)\quad(z\in\bR)
    \ee
    and
    \be\label{sol:boundary}
        \lim_{z\to-\infty}\phi(z;c)=\kappa,\qquad
        \lim_{z\to+\infty}\phi(z;c)=0.
    \ee
    In addition,
    \be\label{sol:strict}
        \phi(z;c)>\phi(0;c)\quad (z<0),\qquad
        \phi(z;c)<\phi(0;c)\quad (z>0).
    \ee
    The following symmetry relation also holds:
    \be\label{sol:symmetry}
        \phi(z;c)=\kappa-\phi(-z;-c)\quad(z,c\in\bR).
    \ee
\end{lemma}
\begin{proof}
    When $c=0$, we have $\phi(z;0)=\cK(z)$.
    Since $D_0^\alpha\phi\equiv0$, all the assertions follow immediately from the properties of $\cK$.
    We therefore assume that $c\neq0$.
    
    Let $\phi(z)=\phi(z;c)$ be a function defined in \eqref{phi}.
    Since $g_{\alpha}$ is positive and $\int^{+\infty}_{0} g_\alpha(r)\,dr=1$,
    the dominated convergence theorem yields \eqref{sol:boundary}.
    Moreover, since $\cK'=-K$, differentiation under the integral sign gives
    \beaa
        \phi'(z) = -\int^{+\infty}_{0} g_\alpha(r)K(z+cr)\,dr\leq0.
    \eeaa
    Thus $\phi$ is non-increasing.

    We next prove \eqref{sol:strict}. 
    For $z<0$,
    \beaa
        \phi(z;c)-\phi(0;c)
        &=& \int^{+\infty}_{0} g_\alpha(r)\left\{\cK(z+cr)-\cK(cr)\right\}dr\\
        &=& \int^{+\infty}_{0} g_\alpha(r) \int_{z+cr}^{cr}K(y)\,dy\,dr \geq0.
    \eeaa
    Since $K(0)>0$ by \eqref{condi:kernel},  
    the inner integral is strictly positive for all $r$ in a sufficiently small interval
    of positive length. 
    Hence, we conclude that
    \beaa
        \phi(z;c)>\phi(0;c)\qquad(z<0).
    \eeaa
    The inequality for $z>0$ follows in the same way.

    We now prove \eqref{sol:eq}. 
    We first consider $c>0$.
    Since $g_\alpha(r)=-f_\alpha'(r)$,
    integration by parts gives
    \bea
        \phi(z)&=& -\int^{+\infty}_{0} f_\alpha'(r)\cK(z+cr)\,dr\nonumber\\
        &=& \cK(z) - c\int^{+\infty}_{0} f_\alpha(r)K(z+cr)\,dr.
    \label{eq:phi-parts-positive}
    \eea
    On the other hand, 
    \beaa
    -\phi'(z+s) = \int^{+\infty}_{0} g_\alpha(q)K(z+s+c q)\,dq.
    \eeaa
    Hence, by Tonelli's theorem and the change of variables $s=c\rho$,
    \beaa
    D_c^\alpha\phi(z)
    &=& \dfrac{c}{\Gamma(1-\alpha)} \int^{+\infty}_{0}\int^{+\infty}_{0} \rho^{-\alpha}g_\alpha(q) K(z+c(\rho+q))\,dq\,d\rho \\
    &=& c\int^{+\infty}_0 K(z+cr) \left\{ \dfrac{1}{\Gamma(1-\alpha)} \int^{r}_0 (r-q)^{-\alpha}g_\alpha(q)\,dq \right\}dr \\
    &=& c\int^{+\infty}_0 f_{\alpha}(r) K(z+cr)dr,
    \eeaa
    where the last identity follows from Lemma~\ref{lem:ML}~(iii).
    By using \eqref{eq:phi-parts-positive}, we have
    \beaa
        D_c^\alpha\phi(z)=\cK(z)-\phi(z).
    \eeaa
    The case for $c<0$ follows in the same way.
    Thus, \eqref{sol:eq} holds for every $c\neq0$.

    Finally, we have
    \beaa
        \cK(x)+\cK(-x)=\kappa.
    \eeaa
    from the fact that $K$ is even.
    Therefore,
    \beaa
        \kappa-\phi(-z;-c)
        &=& \int^{+\infty}_{0} g_\alpha(r) \left\{\kappa-\cK(-z-cr) \right\}dr \\
        &=& \int^{+\infty}_{0} g_\alpha(r)\cK(z+cr)\,dr \\
        &=& \phi(z;c),
    \eeaa
    which proves \eqref{sol:symmetry}.
\end{proof}

\medskip
The next example shows that uniqueness may fail when the speed is zero.

\begin{example}\label{exam:non-unique}
    Suppose that $K$ satisfies \eqref{condi:kernel} and
    \beaa
        \operatorname{supp}K=[-1,1].
    \eeaa
    Assume further that
    \beaa
        2u_T=\kappa.
    \eeaa
    When $c=0$, the profile equation reduces to the
    stationary equation
    \be\label{eq:sta}
        \phi(x) = \int_{\bR} K(x-y)H\bigl(\phi(y)-u_T\bigr)\,dy.
    \ee
    The function
    \beaa
        \phi_1(x):=\cK(x)
    \eeaa
    is a solution of \eqref{eq:sta} satisfying \eqref{boundary}. 
    Indeed,
    \beaa
        \phi_1(x)>u_T\quad (x<0),\qquad
        \phi_1(x)<u_T\quad (x>0).
    \eeaa

    Since $\operatorname{supp}K=[-1,1]$, we have
    \beaa
        \cK(x)=
            \begin{cases}
                0, & x\geq1,\\
                \kappa, & x\leq-1.
            \end{cases}
    \eeaa
    Choose $l_1,l_2\in\bR$ such that
    \beaa
        l_1>2,\qquad l_2-l_1>2,
    \eeaa
    and define
    \beaa
        \phi_2(x) := \cK(x)-\cK(x-l_1)+\cK(x-l_2).
    \eeaa
    Then $\phi_2$ satisfies \eqref{boundary} and
    \beaa
        \phi_2(x)
            \begin{cases}
                >u_T, &x\in(-\infty,0)\cup(l_1,l_2),\\
                =u_T, &x\in\{0,l_1,l_2\},\\
                <u_T, &x\in(0,l_1)\cup(l_2,+\infty).
            \end{cases}
    \eeaa
    Consequently,
    \beaa
        \{\phi_2>u_T\} = (-\infty,0)\cup(l_1,l_2).
    \eeaa
    It follows that
    \beaa
        \int_{\bR}
        K(x-y)H\bigl(\phi_2(y)-u_T\bigr)\,dy 
        &=& \int^{0}_{-\infty} K(x-y)\,dy + \int^{l_2}_{l_1}K(x-y)\,dy \\
        &=& \cK(x)-\cK(x-l_1)+\cK(x-l_2) = \phi_2(x).
    \eeaa
    Thus $\phi_2$ is another solution of \eqref{eq:sta}.
    In particular, the stationary wave profile is not unique in general when $c=0$.
\end{example}

\medskip
To prove uniqueness in the nonzero-speed case, we use the following strong maximum principle.

\begin{lemma}\label{lem:com}
    Let $c\neq 0$.
    Suppose that $v\in BC^1(\bR)$ satisfies 
    \beaa
         \liminf_{x\to\pm\infty} v(x)\ge 0
    \eeaa     
    and
    \beaa
        D^{\alpha}_{c} v + v \ge 0\ \mathrm{on}\ \bR.
    \eeaa
    Then, $v\ge 0$ on $\bR$. 
    Moreover, if there exists $z^*\in\bR$ such that $v(z^*)=0$, then
    \beaa
        v(z)=0\quad (z\geq z^*)
    \eeaa
    when $c>0$, whereas
    \beaa
        v(z)=0\quad(z\leq z^*)
    \eeaa
    when $c<0$.
\end{lemma}
\begin{proof}
    Integration by parts gives 
    \be\label{eq:marchaud}
    D_c^\alpha v(z)
    =
    \frac{\alpha|c|^\alpha}{\Gamma(1-\alpha)}
    \int^{+\infty}_{0}
        \dfrac{v(z)-v\bigl(z+\operatorname{sgn}(c)s\bigr)}{s^{1+\alpha}}\,ds.
    \ee
    The boundary terms vanish because $v$ and $v'$ are bounded and continuous.

    Suppose that $v$ takes a negative value.
    From the assumption for $v$, there exists $z_m\in\bR$ such that
    \beaa
        v(z_m) = \min_{z\in\bR} v(z) <0.
    \eeaa
    Then, we have
    \beaa
        v(z_m)-v\bigl(z_m+\operatorname{sgn}(c)s\bigr)\leq0\qquad(s>0).
    \eeaa
    Thus \eqref{eq:marchaud} gives
    \beaa
        D_c^\alpha v(z_m)\leq0.
    \eeaa
    Consequently,
    \beaa
        D_c^\alpha v(z_m)+v(z_m)<0,
    \eeaa
    which contradicts the assumed differential inequality. 
    Therefore, $v$ is non-negative.
    
    In addition, if there exists $z^*\in\bR$ such that $v(z^*)=0$ then 
    \beaa
         D^{\alpha}_{c} v (z^*) = 0
    \eeaa
    holds from \eqref{eq:marchaud}.
    The integrand in \eqref{eq:marchaud} is
    nonpositive and continuous, and therefore it vanishes identically. 
    Thus
    \beaa
        v\bigl(z^*+\operatorname{sgn}(c)s\bigr)=0
    \eeaa
    for $s>0$.

    The proof is complete.
\end{proof}

\medskip
We are now ready to prove the existence and uniqueness of the
asymptotic traveling wave.

\begin{proof}[Proof of Theorem~\ref{thm:atw}]
    Let $c^*$ be the unique solution of \eqref{eq:speed}. 
    By the definition of $c^*$,
    \beaa
        \phi(0;c^*) = G_{\alpha}(c^*)=u_T.
    \eeaa
    It follows from \eqref{sol:strict} that
    \beaa
        \phi(z;c^*)>u_T\quad (z<0),\qquad
        \phi(z;c^*)<u_T\quad (z>0).
    \eeaa
    Hence $H\bigl(\phi(y;c^*)-u_T\bigr) = H(-y)$ for almost every $y\in\bR$. 
    Therefore,
    \beaa
        \int_{\bR} K(z-y) H\bigl(\phi(y;c^*)-u_T\bigr)\,dy
        = \int^{0}_{-\infty}K(z-y)\,dy 
        = \cK(z).
    \eeaa
    Combining this identity with Lemma~\ref{lem:phi}, we find that 
    $(c^*,\phi(\,\cdot\,;c^*))$ is an asymptotic traveling wave.
    The limits at both infinities, monotonicity, and symmetry relation follow from Lemma~\ref{lem:phi}.

    We next prove uniqueness under the assumption $\kappa\neq2u_T$.
    Let $(c,\psi)$ be an asymptotic traveling wave of Equation~\eqref{eq:main} satisfying \eqref{boundary}.
    Define
    \beaa
        z_{l} &:=& \sup \{ z^*\in\bR\ \mid \ \psi(z)>u_T \ (\forall z<z^*) \}, \\
        z_{r} &:=& \inf \{ z^*\in\bR\ \mid \ \psi(z)<u_T \ (\forall z>z^*) \}.
    \eeaa
    From the continuity of $\psi$, we have $\psi(z_{l}) = \psi(z_r) = u_T$.

    We first show that $c\neq0$.
    If $c=0$, then $\psi$ satisfies
    \beaa
        \psi(z) = \int_{\bR} K(z-y)H(\psi(y)-u_T)\, dy.
    \eeaa
    In particular,
    \beaa
        u_T= \psi(z_l) \geq \phi(0;c) = G_\alpha(c)
    \eeaa
    holds at $z=z_l$. 
    On the other hand, we find
    \beaa
        u_T= \psi(z_r) \leq \phi(0;c) = G_\alpha(c)
    \eeaa
    at $z=z_r$.
    This means $\kappa=2u_T$, which is a contradiction.
    Hence, $c\neq0$.

    Next, let us show the uniqueness of the speed.
    Set $v_{l}(z) := \psi(z) - \phi(z-z_l;c)$.
    From Lemma~\ref{lem:phi}, we get
    \be\label{eq:vl}
        D^{\alpha}_{c} v_l(z) + v_l(z) =\int_{\bR} K(z-y) \{ H(\psi(y) - u_T) - H(z_l-y) \} dy \ge 0
    \ee
    and $v_l(\pm \infty) = 0$.
    Thus, we obtain
    \beaa
        \psi(z) \ge \phi(z-z_l;c)\quad (z\in\bR)
    \eeaa
    from Lemma~\ref{lem:com}.
    In particular, 
    \beaa
        u_T = \psi(z_l)\ge \phi(0;c) =  \int^{+\infty}_{0} g_{\alpha}(r)\cK(c r) dr
    \eeaa
    holds at $z=z_l$.

    Similarly, applying Lemma 4.3 to $v_{r}(z) := \phi(z-z_r;c) - \psi(z)$, we obtain
    \beaa
        \phi(z-z_r;c) \ge \psi(z) \quad (z\in\bR)
    \eeaa
    and
    \beaa
        \int^{+\infty}_{0} g_{\alpha}(r)\cK(c r) dr =  \phi(0;c) \ge u_T.
    \eeaa
    Hence, $c$ is a solution of Equation~\eqref{eq:speed}, which means $c=c^*$.

    Finally, we show the uniqueness of the profile.
    Since the argument is analogous, we only consider the case $c^*>0$.
    We now get
    \beaa
        v_l(z_l) = \psi(z_l) - \phi(0;c^*) = 0.
    \eeaa
    From Lemma~\ref{lem:com}, we have
    \beaa
        \psi(z) = \phi(z-z_l;c^*)\quad (z\ge z_l).
    \eeaa
    Since $\psi(z)>u_T\ (z<z_l)$, \eqref{eq:vl} becomes
    \beaa
        D^{\alpha}_{c} v_l + v_l = 0\quad (z\in\bR).
    \eeaa
    Thus, by applying Lemma~\ref{lem:com} for $v_l$ and $-v_l$,
    we obtain
    \beaa
        \psi(z) = \phi(z-z_l;c^*)\quad (z\in \bR).
    \eeaa

    The proof is complete.
\end{proof}

\medskip
We have constructed an asymptotic traveling wave with speed $c^*$ and
established its uniqueness up to spatial translation when $c^*\neq0$.
The next section is devoted to the convergence of the global front solution to $\phi(\,\cdot\,;c^*)$.

\bigskip
\section{Convergence to asymptotic traveling waves}\label{sec:convergence}

In this section, we prove that the global front solution constructed in Section~\ref{sec:global} converges to the asymptotic traveling wave.
We then present a class of initial data for which the threshold location is exactly linear and an explicit convergence rate is available.

\begin{proof}[Proof of Theorem~2.5]
    We first consider the case $\kappa=2u_T$.
    Note that $\phi(z;0)=\cK(z)$.
    Then, 
    \beaa
        u(t,x) = f_{\alpha}(t) u_0(x) + (1- f_{\alpha}(t))\cK(x)
    \eeaa
    is the front solution to \eqref{eq:main} with \eqref{condi:init}, whose threshold location is $x^*(t)\equiv 0$.
    Then, we have
    \beaa
        |u(t,x^*(t)+z) - \phi(z;0)| &=& f_{\alpha}(t) |u_0(z) - \cK(z)| \le 2\kappa f_{\alpha}(t). 
    \eeaa
    Therefore,
    \beaa
        \dlim_{t\to+\infty} \sup_{z\in\bR}|u(t,x^*(t)+z) - \phi(z;0)| = 0
    \eeaa
    holds. 
    In the case $\kappa=2u_T$, the desired assertion is obtained.

    Next, we consider the case $\kappa\neq 2u_T$.
    From Theorem~\ref{thm:global}, there exists a unique front solution $u(t,x)$ with the threshold location $x^*(t)$.
    By \eqref{condi:kernel}, we obtain
    \beaa
        | \cK(z+x^*(t)-x^*(t-r)) - \cK(z+c^* r)| \le \|K\|_{L^\infty(\bR)} |x^*(t)-x^*(t-r) -c^* r|.
    \eeaa
    Since $u(t,x)$ is represented by \eqref{sol:mild},
    for any $L$, we have 
    \beaa
        |u(t,x^*(t)+z)-\phi(z;c^*)| 
        &\le& f_{\alpha}(t) u_0(x^*(t)+z)\\
        &&\quad + \left[ \int^{L}_{0} + \int^{t}_{L} \right] g_{\alpha}(r) | \cK(z+x^*(t)-x^*(t-r)) - \cK(z+c^* r)|\,dr \\
        &&\quad\quad + \int^{+\infty}_{t} g_{\alpha}(r) \cK(z+c^* r)dr \\
        &\le& \kappa f_{\alpha}(t) +  \|K\|_{L^\infty(\bR)} \sup_{0\le r\le L} |x^*(t)-x^*(t-r) -c^* r|   \\
        &&\quad +2\kappa \int^{+\infty}_{L}  g_{\alpha}(r)\,dr  
    \eeaa
    for every $t>L$.
    By taking the limit as $t\to+\infty$,
    we get
    \beaa
        \limsup_{t\to+\infty} \sup_{z\in\bR} |u(t,x^*(t)+z)-\phi(z;c^*)| \le 2\kappa \int^{+\infty}_{L}  g_{\alpha}(r)\,dr
    \eeaa
    from Proposition~\ref{prop:dif}.
    Since $L$ is arbitrary, we obtain the desired assertion.
\end{proof}
\begin{remark}\label{rem:balanced-rate}
    We note that the convergence rate is obtained in the balanced case $\kappa=2u_T$.
    In this case, the front solution $u(t,x)$ in Theorem~\ref{thm:converge} satisfies $x^*(t)\equiv0$, and the limiting profile is $\phi(\cdot;0)=\cK$.
    As shown in the proof,
    \beaa
        u(t,x)-\cK(x)
        =f_\alpha(t)\bigl(u_0(x)-\cK(x)\bigr)
        \qquad (t\ge0,\ x\in\bR).
    \eeaa
    Consequently,
    \beaa
        \left\|u(t,\cdot)-\cK\right\|_{L^\infty(\bR)}
        = f_\alpha(t) \left\|u_0-\cK\right\|_{L^\infty(\bR)}.
    \eeaa
    By Lemma~\ref{lem:ML}~(i), this yields
    \beaa
        \dlim_{t\to+\infty}t^\alpha \left\|u(t,\cdot)-\cK\right\|_{L^\infty(\bR)} = \dfrac{\left\|u_0-\cK\right\|_{L^\infty(\bR)}}{\Gamma(1-\alpha)}.
    \eeaa
    Thus, the algebraic rate $t^{-\alpha}$ is exact whenever
    $u_0\not\equiv\cK$.
    If $u_0\equiv\cK$, the solution is stationary.
    This conclusion holds for every initial datum satisfying
    \eqref{condi:init}, without any additional decay assumption
    on the connectivity kernel.
\end{remark}

\medskip
We next give a class of front-like initial data for which the threshold location moves exactly with the non-zero asymptotic speed.
\begin{proposition}\label{prop:ct}
    Assume that Conditions \eqref{condi:kernel} and \eqref{condi:bistable} hold.
    Suppose that $\kappa>2u_T$.
    Let $u_0$ be a bounded, continuous and non-increasing function satisfying
    \be\label{init:right}
        u_0(x) = \dfrac{\dint^{+\infty}_{x/c^*} g_{\alpha}(r) \cK(c^* r)dr}{f_{\alpha}(x/c^*)}\quad (x\ge 0)
    \ee
    and
    \be\label{init:left}
        u_0(x)>u_T\quad (x<0).
    \ee
    Then, there exists a front solution $u(t,x)$ with the threshold location $x^*(t)=c^*t$, where $c^*$ is the unique solution of \eqref{eq:speed}.
    Moreover, it satisfies
    \beaa
        \sup_{z\in\bR} |u(t,x^*(t)+z) - \phi(z;c^*)| \le (\|u_0\|_{L^\infty}+\kappa) f_{\alpha}(t).
    \eeaa
\end{proposition}
\begin{proof}
    Define
    \beaa
        u(t,x):= f_\alpha(t)u_0(x) + \int^t_{0} g_\alpha(t-s)\cK(x-c^*s)\,ds.
    \eeaa
    We first show that $u(t,x)$ is a front solution to \eqref{eq:main}.
    For \(z\neq0\),
    \beaa
        u(t,c^*t+z)-u(t,c^*t)&=& f_\alpha(t) \{u_0(c^*t+z)-u_0(c^*t)\} \\
        &&\qquad + \int^t_{0} g_\alpha(r) \{\cK(z+c^*r)-\cK(c^*r)\}\,dr.
    \eeaa
    Since $f_\alpha(t)>0$ and both $u_0$ and $\cK$ are non-increasing, 
    $u(t,x)$ is non-increasing with respect to $x$ for every $t\ge 0$.

    By the definition of $u_0$, we have
    \beaa
        f_\alpha(t)u_0(c^*t) = \int^{+\infty}_{t} g_\alpha(r)\mathcal K(c^*r)\,dr.
    \eeaa
    Then, we obtain
    \beaa
        u(t,c^*t)&=& f_\alpha(t)u_0(c^*t)
        + \int^t_{0} g_\alpha(t-s) \cK\bigl(c^*(t-s)\bigr)\,ds \\
        &=& \int^{+\infty}_{t} g_\alpha(r)\mathcal K(c^*r)\,dr
        + \int^t_0 g_\alpha(r)\mathcal K(c^*r)\,dr\\
        &=& \int^\infty_{0} g_\alpha(r)\mathcal K(c^*r)\,dr =u_T,
    \eeaa
    where the last equality follows from Equation~\eqref{eq:speed}.
    By using \eqref{init:left}, \eqref{init:right} and $K(0)>0$,
    it holds that
    \beaa
        u(t,x)>u_T\quad (x<c^*t),\qquad u(t,x)<u_T\quad (x>c^*t).
    \eeaa
    Consequently,
    \beaa
        H\bigl(u(s,y)-u_T\bigr)=H(c^*s-y)
    \eeaa
    for almost every $y\in\mathbb R$, and hence
    \beaa
        \int_{\bR} K(x-y)H\bigl(u(s,y)-u_T\bigr)\,dy
        = \int_{-\infty}^{c^*s}K(x-y)\,dy
        = \cK(x-c^*s).
    \eeaa
    Thus $u(t,x)$ coincides with the mild formulation of the Cauchy problem. 
    Therefore, $u$ is a front solution with threshold location
    $x^*(t)=c^*t$.

    It remains to prove the convergence estimate. By the representation of $\phi(\,\cdot\,;c^*)$, we have
    \beaa
        u(t,c^*t+z)-\phi(z;c^*)
        &=& f_\alpha(t)u_0(c^*t+z) - \int^\infty_{t} g_\alpha(r)\cK(z+c^*r)\,dr.
    \eeaa
    Since $0\leq\cK\leq\kappa$, it follows from Lemma~\ref{lem:ML}~(ii) that
    \beaa
        \left| u(t,c^*t+z)-\phi(z;c^*) \right|
        &\leq& f_\alpha(t)\|u_0\|_{L^\infty(\bR)}+\kappa\int^{+\infty}_{t} g_\alpha(r)\,dr\\
        &=& \left( \|u_0\|_{L^\infty(\bR)} + \kappa \right)f_\alpha(t).
    \eeaa
    Taking the supremum over $z\in\bR$ proves the desired estimate.
\end{proof}
\begin{remark}
    The function $u_0$ defined by \eqref{init:right} satisfies
    \beaa
        u_0(0)=u_T,\qquad \lim_{x\to+\infty}u_0(x)=0.
    \eeaa
    Moreover, $u_0$ is non-increasing on $[0,\infty)$.
    Finally,
    \beaa
        u_0(x)-u_T =\dfrac{u_T-\kappa/2}{(c^*)^\alpha\Gamma(1+\alpha)} x^\alpha + o(x^\alpha)
    \eeaa
    as $x\to+0$.
    Since $0<\alpha<1$ and $\kappa>2u_T$, the function $u_0$ is not differentiable at the origin.
\end{remark}

\medskip
The results of this section show that the asymptotic traveling wave with speed $c^*$ is realized as the long-time profile of the global front solution. 
For the initial data considered in Proposition~\ref{prop:ct}, the threshold location is exactly $c^*t$, and the convergence is of order $O(t^{-\alpha})$.
For nonzero speeds, we expect the algebraic rate $t^{-\alpha}$ to be sharp for general front-like initial data satisfying \eqref{condi:init}.
However, neither a general $O(t^{-\alpha})$ upper bound nor a corresponding lower bound is established here.

\bigskip
\section{Kernel-dependent tail asymptotics of the wave profile}\label{sec:ker}

In this section, we investigate how the decay of the connectivity kernel determines the tail behavior of the wave profile \eqref{phi}.
When $c=0$, we simply have
\beaa
    \phi(z;0)=\cK(z).
\eeaa
Moreover, the symmetry relation \eqref{inverse} shows that it is sufficient to consider $c>0$.
The corresponding results for $c<0$ follow by exchanging the left- and right-hand tails.

For $c>0$, the monotonicity of $\cK$ and
$\int_0^\infty g_\alpha(r)\,dr=1$ imply
\beaa
    0 \le \phi(z;c) \le \cK(z) = \phi(z;0).
\eeaa
We first determine the precise behavior of the right-hand tail.
We consider three representative classes of connectivity kernels:
algebraically decaying, exponentially decaying, and super-exponentially
decaying kernels.

\begin{proposition}\label{prop:right-decay}
    Suppose that $c\ge 0$.
    Under the condition \eqref{condi:kernel}, the following statements hold:
    \begin{itemize}
        \item[(i)] Suppose that there exist $\beta>1$ and $C_{alg}\ge 0$ such that
        \be\label{kernel:alg}
            \dlim_{x\to+\infty} x^{\beta} K(x) = C_{alg}. \tag{K1}
        \ee
        Then, $\phi(z;c)$ satisfies
        \beaa
            \dlim_{z\to +\infty} z^{\beta-1} \phi(z;c) = \dfrac{C_{alg}}{\beta-1}.
        \eeaa
        
        \item[(ii)] Suppose that there exist $\nu>0$ and $C_{exp}\ge 0$ such that
        \be\label{kernel:exp}
            \dlim_{x\to+\infty} e^{\nu x} K(x) = C_{exp}. \tag{K2}
        \ee
        Then, $\phi(z;c)$ satisfies
        \beaa
            \dlim_{z\to +\infty} e^{\nu z} \phi(z;c) = \dfrac{C_{exp}}{\nu\{1+(\nu c)^{\alpha}\}}.
        \eeaa

        \item[(iii)] Suppose that there exist $\sigma>0$, $p>1$, and $C_{sup}> 0$ such that
        \be\label{kernel:sup}
            \dlim_{x\to+\infty} e^{\sigma x^p} K(x) = C_{sup}. \tag{K3}
        \ee
        Then, $\phi(z;c)$ satisfies
        \beaa
            \dlim_{z\to+\infty} z^{p-1} e^{\sigma z^p} \phi(z;0) = \dfrac{C_{sup}}{\sigma p}
        \eeaa
        when $c=0$, and        
        \beaa
            \lim_{z\to+\infty} z^{(1+\alpha)(p-1)} e^{\sigma z^p}\phi(z;c) = \frac{C_{sup}} {(\sigma p)^{1+\alpha}c^\alpha}
        \eeaa
        when $c>0$.
    \end{itemize}
\end{proposition}

\medskip
Proposition~\ref{prop:right-decay} shows that the effect of the time-fractional derivative on the right-hand tail depends on the decay scale of the connectivity kernel.

We next consider the left-hand tail. 
In contrast to the right-hand tail,
the decay of $\kappa-\phi(z;c)$ is determined by a competition between
the algebraic tail of the connectivity kernel and the long memory induced
by the Caputo derivative.

\begin{proposition}\label{prop:left-decay}
    Suppose that $c>0$.
    Under the conditions \eqref{condi:kernel} and \eqref{kernel:alg}, we have
        \beaa
        \dlim_{z\to -\infty} |z|^{\beta-1}(\kappa -\phi(z;c)) = \dfrac{C_{alg}}{\beta-1}
        \eeaa
        when $\beta<1+\alpha$, 
        \beaa
        \dlim_{z\to -\infty} |z|^{\alpha}(\kappa -\phi(z;c)) = \dfrac{C_{alg}}{\beta-1} + \dfrac{\kappa c^{\alpha}}{\Gamma(1-\alpha)}
        \eeaa
        when $\beta=1+\alpha$, and
        \be\label{lim:atw}
        \dlim_{z\to -\infty} |z|^{\alpha}(\kappa -\phi(z;c)) = \dfrac{\kappa c^{\alpha}}{\Gamma(1-\alpha)}
        \ee
        when $\beta>1+\alpha$.
\end{proposition}

\medskip
Proposition~\ref{prop:left-decay} exhibits a transition at $\beta=1+\alpha$.
When $C_{alg}>0$ and $1<\beta<1+\alpha$, the connectivity kernel decays sufficiently slowly that its algebraic tail determines the leading behavior of $\kappa-\phi(z;c)$.
When $\beta>1+\alpha$, the contribution of the kernel is of lower order, and the fractional memory produces the universal
tail \eqref{lim:atw}.
At the critical exponent $\beta=1+\alpha$, the kernel contribution, when $C_{\rm alg}>0$, is of the same order as the memory contribution, and the two coefficients add.

In particular, if the connectivity kernel decays exponentially or super-exponentially, then \eqref{lim:atw} holds.
Thus, even when the connectivity kernel has a rapidly decaying tail, the wave profile generally possesses an algebraic tail on the side behind the front. 
This asymmetric behavior is a direct consequence of the long-memory effect of the time-fractional derivative.

The proofs of Propositions~\ref{prop:right-decay} and \ref{prop:left-decay} are deferred to Appendix~\ref{app:ker}.

\bigskip
\section{Concluding remarks}\label{sec:dis}

In this paper, we have developed a rigorous framework for front propagation in a time-fractional Amari model.
The Heaviside firing-rate function reduces the Cauchy problem to a nonlinear Volterra equation for the threshold location. 
Within this framework, we constructed a unique global front solution, identified its short- and long-time behavior, and showed that the asymptotic traveling wave is realized as the long-time limit of the solution. 
A central ingredient in the convergence analysis is the asymptotic linearity of the local increments of the threshold location, which is obtained through a rigidity argument for its translation limits. 
The explicit representation of the profile also makes it possible to relate its tail behavior to the decay properties of the connectivity kernel.

Several questions remain open. 
First, our main convergence theorem concerns strictly decreasing front-like initial data.
A natural problem is to characterize a wider class of initial conditions for the convergence of the asymptotic traveling wave. 
This includes front-like initial data that are not monotone, as well as localized perturbations of monotone fronts. 
Even within the class considered here, a quantitative convergence rate in the nonzero-speed case remains to be established for general initial data.
Determining the second-order asymptotics of the threshold location also remains an open problem.

Second, the formulation by the threshold location used in this paper relies essentially on the Heaviside firing-rate function. 
Extending the present approach to smooth sigmoidal firing rates would require a different mechanism for tracking the transition layer. 
Establishing the persistence of speed selection and convergence to asymptotic traveling waves under such regularizations is an important direction for future work.

\bigskip
\section*{Acknowledgments}
HI is supported by JSPS KAKENHI Grant Numbers 23K13013, 24H00188 and 25K07142.

\bigskip
\section*{Declaration of generative AI and AI-assisted technologies in the manuscript preparation process}
During the preparation of this work, the author used ChatGPT-5.6 sol for language editing and consistency checking.
The author reviewed and edited the output as needed and takes full responsibility for the content of the published article.

\appendix
\numberwithin{equation}{section}

\bigskip
\section{Proofs of Propositions~\ref{prop:right-decay} and \ref{prop:left-decay}}\label{app:ker}

In this appendix, we prove the tail asymptotics stated in Propositions~\ref{prop:right-decay} and~\ref{prop:left-decay}.

\bigskip
\begin{proof}[Proof of Proposition~\ref{prop:right-decay}]
    Assume that $c\geq 0$ and that Condition~\eqref{condi:kernel} holds.

    (i) By l'H\^{o}pital's rule and \eqref{kernel:alg}, we obtain
    \be\label{cK:alg}
         \dlim_{x\to +\infty} x^{\beta-1} \cK(x)   = \dfrac{C_{alg}}{\beta-1}.
    \ee
    Since $\phi(z;0)=\cK(z)$, this proves the assertion when $c=0$.

    Let $c>0$. The dominated convergence theorem yields
    \beaa
        z^{\beta-1} \phi(z;c) &=& \int^{+\infty}_{0} g_{\alpha}(r) \dfrac{z^{\beta-1}}{(z+cr)^{\beta-1}} \left( (z+cr)^{\beta-1} \cK(z+cr) \right)\,dr \\
        &\to& \dfrac{C_{alg}}{\beta-1} \int^{+\infty}_{0} g_{\alpha}(r)\, dr = \dfrac{C_{alg}}{\beta-1}
    \eeaa
    as $z\to+\infty$.
    
    (ii) By l'H\^{o}pital's rule and \eqref{kernel:exp},
    \beaa
         \dlim_{x\to +\infty} e^{\nu x} \cK(x) = \dfrac{C_{exp}}{\nu}.
    \eeaa
    This immediately proves the assertion when $c=0$.

    Let $c>0$.
    The dominated convergence theorem yields
    \beaa
        e^{\nu z} \phi(z;c) &=& \int^{+\infty}_{0} g_{\alpha}(r)  e^{-c\nu r} \left( e^{\nu(z+cr)} \cK(z+cr) \right)\,dr \\
        &\to& \dfrac{C_{exp}}{\nu} \int^{+\infty}_{0} g_{\alpha}(r)  e^{-c\nu r}\, dr = \dfrac{C_{exp}}{\nu\{1+(c\nu)^{\alpha}\}}
    \eeaa
    as $z\to+\infty$.
    Here, we used \cite[formula (7.1)]{HMS} in the last equality.

    (iii) By l'H\^{o}pital's rule and \eqref{kernel:sup}, we first obtain
    \be\label{cK:sup} 
        \dlim_{z\to+\infty} z^{p-1} e^{\sigma z^p} \phi(z;0) = \dlim_{x\to+\infty} x^{p-1} e^{\sigma x^p} \cK(x) = \dfrac{C_{sup}}{\sigma p}. 
    \ee 
    Indeed, 
    \beaa
        \lim_{x\to+\infty} \frac{\cK(x)}{x^{1-p}e^{-\sigma x^p}} &=& \lim_{x\to+\infty} \frac{-K(x)} {e^{-\sigma x^p} \{(1-p)x^{-p}-\sigma p\}} \\ 
        &=& \frac{C_{\rm sup}}{\sigma p}. 
    \eeaa
    This proves the assertion when $c=0$.

    We next consider the case $c>0$. 
    Set $\lambda(z):=\sigma pc\,z^{p-1}$.  
    Then $\lambda(z)\to+\infty$ as $z\to+\infty$, and the change of variables $s=\lambda(z)r$ gives
    \beaa
        \phi(z;c) &=& \dfrac{1}{\lambda(z)} \int^{+\infty}_{0} g_\alpha\left(\dfrac{s}{\lambda(z)}\right) \cK\left( z+\dfrac{cs}{\lambda(z)} \right)\,ds. 
    \eeaa
    For each fixed $s>0$, 
    the expansion of the Mittag--Leffler function yields
    \be\label{integrand1} 
        \dlim_{z\to+\infty}\lambda(z)^{\alpha-1} g_\alpha\left(\dfrac{s}{\lambda(z)}\right)= \frac{s^{\alpha-1}}{\Gamma(\alpha)}.
    \ee 
    Moreover, \eqref{cK:sup} implies
    \be\label{integrand2}
        \dlim_{z\to +\infty}\dfrac{ \cK\left(z+\dfrac{cs}{\lambda(z)}\right)} {\cK(z)} = e^{-s}.
    \ee
    
    Combining \eqref{integrand1} and \eqref{integrand2}, we obtain
    \beaa
    \dlim_{z\to+\infty} \dfrac{\{\lambda(z)\}^{\alpha} \phi(z;c)}{\cK(z)} &=& \int^{+\infty}_{0} \dlim_{z\to+\infty} \left( \lambda(z)^{\alpha-1} g_\alpha\left(\frac{s}{\lambda(z)}\right) \dfrac{ \cK\left(z+\dfrac{cs}{\lambda(z)}\right)}{\cK(z)} \right) \,ds \\
    &=& \dfrac{1}{\Gamma(\alpha)} \int^{+\infty}_{0} s^{\alpha-1} e^{-s}\,ds = 1.
    \eeaa    
    The proof is complete.
\end{proof}

\medskip
Next, we give the proof of Proposition~\ref{prop:left-decay}.

\begin{proof}[Proof of Proposition~\ref{prop:left-decay}]
    Let $c>0$. Since $K$ is even, we have
    \beaa
        \cK(-x)=\kappa-\cK(x).
    \eeaa
    Thus, 
    \beaa
        \kappa-\phi(z;c)
        = \int^{+\infty}_0 g_\alpha(r) \bigl\{\kappa-\cK(z+cr)\bigr\}\,dr
        = \int^{+\infty}_{0} g_\alpha(r)\cK(-z-cr)\,dr.
    \eeaa

    Since $g_{\alpha}=-f'_{\alpha}$ and $\cK'=-K$,
    integration by parts gives 
    \bea
        \kappa-\phi(z;c) &=& \cK(-z) + c\int^{+\infty}_{0} f_\alpha(r)K(-z-cr)\,dr \notag \\
        &=& \cK(-z) +
        \int^{-z}_{-\infty} f_\alpha \left(\frac{-z-y}{c}\right)K(y)\,dy =  \cK(-z) + J(-z),\label{kappa-phi}
    \eea
    where
    \beaa
        J(z) &=& J_1(z)+J_2(z)
    \eeaa
    with
    \beaa
        J_1(z) := \int_{-\infty}^{z/2} f_\alpha \left(\frac{z-y}{c}\right)K(y)\,dy, \quad
        J_2(z) := \int_{z/2}^{z}
        f_\alpha \left( \frac{z-y}{c}\right)K(y)\,dy.
    \eeaa
    
    We first evaluate $J_1(z)$. 
    By Lemma~\ref{lem:ML}~(i), for every fixed $y\in\bR$,
    \beaa
        \dlim_{z\to+\infty} z^{\alpha} f_\alpha \left(\frac{z-y}{c}\right) = \dfrac{c^\alpha}{\Gamma(1-\alpha)}.
    \eeaa
    Thus, the dominated convergence theorem yields 
    \beaa
        \lim_{z\to+\infty} z^\alpha J_1(z) = \dfrac{\kappa c^\alpha}{\Gamma(1-\alpha)}.
    \eeaa

    We next estimate $J_2(z)$.
    By \eqref{kernel:alg}, there exist $R>0$ and $C_K>0$ such that
    \beaa
        K(x) \le  C_K |2x|^{-\beta}\quad (x\ge R).
    \eeaa
    For $z\ge 2R$,
    \beaa
        J_2(z)&\le& C_K z^{-\beta} \int_{z/2}^{z} E_\alpha \left(- \left(\frac{z-y}{c}\right)^{\alpha}\right)\,dy \\
        &=& cC_K z^{-\beta}
        \int^{z/(2c)}_{0} E_\alpha(-r^{\alpha})\,dr\\
        &=&  O(z^{1-\alpha-\beta})
    \eeaa
    holds. We thus get
    \beaa
        \dlim_{z\to+\infty} z^\alpha J_2(z) =0.
    \eeaa
    Hence, we obtain
    \be\label{lim:J}
        \dlim_{z\to+\infty} z^{\alpha} J(z) = \frac{\kappa c^\alpha}{\Gamma(1-\alpha)}.
    \ee

    By using \eqref{kappa-phi}, \eqref{cK:alg} and \eqref{lim:J}, 
    we obtain 
    \beaa
        \dlim_{z\to -\infty} |z|^{\beta-1}(\kappa -\phi(z;c)) = \dlim_{z\to+\infty} z^{\beta-1}(\cK(z) + J(z)) = \dfrac{C_{alg}}{\beta-1}
    \eeaa
    when $\beta<1+\alpha$.
    Similarly, we have
    \beaa
        \dlim_{z\to -\infty} |z|^{\alpha}(\kappa -\phi(z;c)) = \dlim_{z\to+\infty} z^{\alpha}(\cK(z) + J(z)) = \dfrac{C_{alg}}{\beta-1} + \frac{\kappa c^\alpha}{\Gamma(1-\alpha)}
    \eeaa
    when $\beta=1+\alpha$, and
    \beaa
        \dlim_{z\to -\infty} |z|^{\alpha}(\kappa -\phi(z;c)) = \dlim_{z\to+\infty} z^{\alpha}(\cK(z) + J(z)) =  \frac{\kappa c^\alpha}{\Gamma(1-\alpha)}
    \eeaa
    when $\beta>1+\alpha$.

    The proof is complete.
\end{proof}

\bigskip

\end{document}